\documentclass[letterpaper,11pt]{article}
\usepackage{setspace} 
\usepackage{verbatim}

\usepackage[color=green!40, textsize=small]{todonotes}

\usepackage{float}
\usepackage[dvipsnames,svgnames,table]{xcolor}
\usepackage[colorlinks=true,linkcolor=blue,citecolor=ForestGreen]{hyperref}

\usepackage{array}
\usepackage{xspace}

\usepackage{geometry}
\usepackage{amsmath,amsthm,amssymb}
\usepackage[capitalise]{cleveref}

\usepackage{orcidlink}

\newtheorem{theorem}{Theorem}[section]

\newtheorem{corollary}[theorem]{Corollary}
\newtheorem{observation}[theorem]{Observation}

\newtheorem{lemma}[theorem]{Lemma}
\theoremstyle{definition}
\newtheorem{definition}[theorem]{Definition}
\theoremstyle{remark}
\newtheorem{remark}[theorem]{Remark}

\usepackage{graphicx} 
\usepackage{xspace}

\newcommand{\cB}{\ensuremath{\mathcal{B}}}
\newcommand{\cC}{\ensuremath{\mathcal{C}}}
\newcommand{\cD}{\ensuremath{\mathcal{D}}}
\newcommand{\cE}{\ensuremath{\mathcal{E}}}
\newcommand{\cF}{\ensuremath{\mathcal{F}}}

\newcommand{\cI}{\ensuremath{\mathcal{I}}}

\newcommand{\cL}{\ensuremath{\mathcal{L}}}

\newcommand{\cP}{\ensuremath{\mathcal{P}}}
\newcommand{\cQ}{\ensuremath{\mathcal{Q}}}
\newcommand{\cR}{\ensuremath{\mathcal{R}}}
\newcommand{\cS}{\ensuremath{\mathcal{S}}}
\newcommand{\cT}{\ensuremath{\mathcal{T}}}

\newcommand{\cY}{\ensuremath{\mathcal{Y}}}
\newcommand{\cX}{\ensuremath{\mathcal{X}}}
\newcommand{\cZ}{\ensuremath{\mathcal{Z}}}

\newcommand{\XY}{\mathcal{\mathcal{D}}^*}

\newcommand{\bN}{\ensuremath{\mathbb{N}}}

\newcommand{\zo}{\{0,1\}}

\newcommand\xor{\oplus}
\newcommand\Xor{\bigoplus}

\newcommand{\EH}{Erd\H{o}s-Hajnal\xspace}

\DeclareMathOperator{\dgn}{dgn}
\DeclareMathOperator{\sch}{sch}
\DeclareMathOperator{\ch}{ch}

\DeclareMathOperator{\degree}{deg}

\usepackage{soul}

  \newcommand{\SuggestChange}[2]{{\color{red} \relax\ifmmode\text{\st{$#1$}}\else \st{#1}\fi}\ {\color{blue} #2}}

\tikzset{
    emptygraph/.pic={
        \foreach \i in {1, 2, 3} {
            \node[circle, draw, fill=black, inner sep=1pt] (col1-\i) at (0, -\i*0.5) {};
        }
    
        \foreach \i in {1, 2, 3} {
            \node[circle, draw, fill=black, inner sep=1pt] (col2-\i) at (0.75, -\i*0.5) {};
        }
    
        \node at (1.45, -1) {\dots};
    
        \foreach \i in {1, 2, 3} {
            \node[circle, draw, fill=black, inner sep=1pt] (col3-\i) at (2, -\i*0.5) {};
        }
    }  
}

\tikzset{
    oneedgegraph/.pic={
        \node[circle, draw, fill=black, inner sep=1pt] (col1-1) at (0.2, -0.65) {};
        \node[circle, draw, fill=black, inner sep=1pt] (col1-2) at (0.2, -1.35) {};
        \draw (col1-1) -- (col1-2); 
    
        \foreach \i in {1, 2, 3} {
            \node[circle, draw, fill=black, inner sep=1pt] (col2-\i) at (0.75, -\i*0.5) {};
        }
    
        \node at (1.45, -1) {\dots};
    
        \foreach \i in {1, 2, 3} {
            \node[circle, draw, fill=black, inner sep=1pt] (col3-\i) at (2, -\i*0.5) {};
        }
    }  
}

\tikzset{
    cliqueplusvertex/.pic={
        \draw[thick] (0, 0) circle (0.7cm);
    
        \draw[thick, rotate around={45:(0,0)}] (-0.7, 0) -- (0.7, 0); 
        \draw[thick, rotate around={45:(0,0)}] (0, -0.7) -- (0, 0.7); 
    
        \node[circle, draw, fill=black, inner sep=1pt] at (1.3, 0) {};
    }
}

\tikzset{
    cliqueplusis/.pic={
        \draw[thick] (0, 0) circle (0.7cm);
    
        \draw[thick, rotate around={45:(0,0)}] (-0.7, 0) -- (0.7, 0); 
        \draw[thick, rotate around={45:(0,0)}] (0, -0.7) -- (0, 0.7); 
    
        \foreach \i in {1, 2, 3} {
            \node[circle, draw, fill=black, inner sep=1pt] (col2-\i) at (1.25, -\i*0.5+1) {};
        }
    
        \node at (1.95, 0) {\dots};
    
        \foreach \i in {1, 2, 3} {
            \node[circle, draw, fill=black, inner sep=1pt] (col3-\i) at (2.5, -\i*0.5+1) {};
        }
    }
}

\tikzset{
    degone/.pic={
        \node[circle, draw, fill=black, inner sep=1pt] (col1-1) at (0, 0.35) {};
        \node[circle, draw, fill=black, inner sep=1pt] (col1-2) at (0, -0.35) {};
        \draw (col1-1) -- (col1-2); 

        \node[circle, draw, fill=black, inner sep=1pt] (col2-1) at (0.75, 0.35) {};
        \node[circle, draw, fill=black, inner sep=1pt] (col2-2) at (0.75, -0.35) {};
        \draw (col2-1) -- (col2-2); 

        \node at (1.55, 0) {\dots};

        \node[circle, draw, fill=black, inner sep=1pt] (col3-1) at (2.2, 0.35) {};
        \node[circle, draw, fill=black, inner sep=1pt] (col3-2) at (2.2, -0.35) {};
        \draw (col3-1) -- (col3-2); 
    }  
}

\tikzset{
    equivgraph/.pic={
        \draw[thick] (-1.8, 0) circle (0.7cm);
    
        \draw[thick, rotate around={45:(-1.8,0)}] (-2.5, 0) -- (-1.1, 0); 
        \draw[thick, rotate around={45:(-1.8,0)}] (-1.8, -0.7) -- (-1.8, 0.7); 

        \draw[thick] (0, 0) circle (0.7cm);
    
        \draw[thick, rotate around={45:(0,0)}] (-0.7, 0) -- (0.7, 0); 
        \draw[thick, rotate around={45:(0,0)}] (0, -0.7) -- (0, 0.7); 
    
        \node at (1.5, 0) {\dots};
    
        \draw[thick] (3, 0) circle (0.7cm);
    
        \draw[thick, rotate around={45:(3,0)}] (2.3, 0) -- (3.7, 0); 
        \draw[thick, rotate around={45:(3,0)}] (3, -0.7) -- (3, 0.7); 
    }
}

\title{Boolean combinations of graphs}
\author{Sarosh Adenwalla\thanks{Department of Computer Science, University of Liverpool, UK, \texttt{sarosh.adenwalla@liverpool.ac.uk}, \orcidlink{0009-0009-8582-1281}} 
\and 
Samuel Braunfeld\thanks{Charles University, Computer Science Institute, Malostransk\'{e} N\'{a}m\v{e}st\'{i} 25, 150 00 Prague, Czech Republic; and The Czech Academy of Sciences, Institute of Computer Science, Pod Vod\'{a}renskou v\v{e}\v{z}\'{\i} 2, 182 00 Prague, Czech Republic, \texttt{sbraunfeld@iuuk.mff.cuni.cz}, \orcidlink{0000-0003-3531-9970}} 
\and 
John Sylvester\thanks{Department of Computer Science, University of Liverpool, UK, \texttt{john.sylvester@liverpool.ac.uk}, \orcidlink{0000-0002-6543-2934}} 
\and 
Viktor Zamaraev\thanks{Department of Computer Science, University of Liverpool, UK, \texttt{viktor.zamaraev@liverpool.ac.uk}, \orcidlink{0000-0001-5755-4141}}}
\date{}

\begin{document}

\maketitle

\begin{abstract}
    Boolean combinations allow combining given combinatorial objects to obtain new, potentially more complicated, objects. In this paper, we initiate a systematic study of this idea applied to graphs. In order to understand expressive power and limitations of Boolean combinations in this context, we investigate how they affect different combinatorial and structural properties of graphs, in particular $\chi$-boundedness, as well as characterize the structure of Boolean combinations of graphs from various classes.
\end{abstract}

\newpage

\setcounter{tocdepth}{2}
\tableofcontents

\newpage

\section{Introduction}

Understanding complexity of graphs with respect to combinatorial, structural, or algorithmic properties is a recurring theme in graph theory.
A broad approach to measuring the complexity of graphs is to express them using a small number of simpler graphs.
This includes modular decomposition \cite{Gal67}, graph coverings \cite{Bei68}, edge-decompositions \cite{Mer16}, tree decomposition \cite{Hal76,RS84}, and logical transductions \cite{CE12}.

Graph coverings and edge-decompositions are examples of how graphs can be expressed as unions of other graphs.
The corresponding measures of complexity of a graph $G$
with respect to a class $\cX$ are defined as the minimum number of graphs in $\cX$ whose union or edge-disjoint union coincides with $G$.
These include clique cover number (with respect to cliques) \cite{EGP66,Alon86},
threshold dimension (wrt threshold graphs) \cite{CH77}, 
thickness (wrt planar graphs) \cite{MOS98},
equivalence covering number (wrt equivalence graphs) \cite{Duc79,Alon86}, arboricity (wrt forests) \cite{NW64}, linear arboricity (wrt linear forests) \cite{Har70}, star arboricity (wrt star forests) \cite{AK84,AA89}, biclique cover number or bipartite dimension (wrt complete bipartite graphs) \cite{FH96}, biparticity (wrt bipartite graphs) \cite{HHM77}.

Similarly, graphs can be expressed as intersections of several graphs with certain properties. Complexity measures associated with such representations are known as intersection dimensions.
The intersection dimension of a graph $G$ with respect to a graph class $\cX$ is the minimum number $k$ such that $G$ is the intersection of $k$ graphs from $\cX$ \cite{CozzensR89,KratochvilT94}. Intersection dimensions of graphs with respect to specific graph classes have been studied extensively 
and can be encountered under different names, including boxicity (with respect to interval graphs) and cubicity (wrt unit interval graphs) \cite{Rob69}, 
circular dimension (wrt circular-arc graphs) \cite{Fei79, She80}, overlap dimension (wrt circle graphs \cite{CozzensR89}),
chordality (wrt chordal graphs) \cite{McKeeS93}, Ferrers dimension (wrt chain graphs) \cite{ChaplickHOSU17}, threshold dimension (wrt threshold graphs) and cograph dimension (wrt cographs) \cite{ChackoF21}.

Another related type of graph representation is the expression of graphs as XORs (sum mod 2) of several simpler graphs.
XORs of complete bipartite graphs \cite{BCCNORY23}, graphs consisting of a clique and isolated vertices \cite{BuchananPR22,PouzetKT21}, paths \cite{BBCFRY23} have been studied in the literature.

A natural generalization of the above three types of graph representations (i.e., unions, intersections, and XORs) are expressions of graphs as arbitrary Boolean combinations of several simpler graphs.
Boolean combinations of graphs have recently proved useful in the context of adjacency labeling schemes\footnote{We postpone the definitions of most terms mentioned in the introduction until the next section.} \cite{Cha23, HZ24}. This is due to the  simple fact that if graphs from a class can be expressed as Boolean combinations of graphs from another class, then an adjacency labeling scheme for the latter class can naturally be extended to an adjacency labeling scheme for the former class and the size of adjacency labels is increased only by a constant factor.

Motivated by this application, the main goal of this paper is to initiate a systematic study of Boolean combinations of graphs with the aim of understanding their expressive power and limitations. In order to do this, we study how Boolean combinations affect different combinatorial and structural graph properties, as well as characterize the structure of Boolean combinations of graphs from various classes.

To outline our results, we first define informally Boolean functions of graphs and graph classes.
Let $G, H_1, H_2, \ldots, H_k$ be graphs on the same vertex set $V$. 
We say that $G$ is a \emph{Boolean combination} or a \emph{Boolean function of}  $H_1, H_2, \ldots, H_k$, if there exists a $k$-variable Boolean function $f$ such that for all $a,b \in V$, the adjacency between $a,b \in V$ in $G$ is determined by applying $f$ to the adjacencies between $a,b$ in the graphs $H_1, H_2, \ldots, H_k$.
This definition naturally extends to classes of graphs. 
Namely, a class $\cX$ is a \emph{function} of a class $\cY$, if there exists $k \in \bN$ such that every graph in $\cX$ is a Boolean combination of $k$ graphs in $\cY$.

\paragraph{Boolean combinations and properties of graph classes.}
In \cref{sec:properties}, we consider several properties of graph classes and show how they are affected by Boolean combinations. 
These include the speed of a class, the size of adjacency labeling schemes, the \EH property, edge-stability, the 
neighborhood complexity, and VC dimension.
Each result in this section is either fairly straightforward or a reformulation of a known fact in our context, but they seem worth collecting in one place and serve as a warm-up for how Boolean combinations can affect properties of graph classes.

Naturally, not every property is preserved by Boolean combinations. For example,  monadic stability, monadic dependence, and $\chi$-boundedness, are not preserved. The first two of these properties, together with a broader model-theoretic perspective on Boolean combinations of graphs, are discussed in \cref{sec:model-theory}.
The effect of Boolean combinations on $\chi$-boundedness is studied in detail in  \cref{sec:chi-boundedness} and further discussed later in this section.

\paragraph{Boolean functions of special graph classes.} 

In \cref{sec:special-graph-classes}, we provide characterizations of Boolean combinations of some special classes of graphs.
In \cref{sec:monotone-classes}, we characterize Boolean combinations of graph classes of bounded degree, bounded degeneracy, bounded chromatic number, and weakly-sparse graph classes.
In \cref{sec:equivalence-graphs}, we consider a hierarchy of subclasses of equivalence graphs (i.e., graphs in which every component is a clique) and characterize Boolean combinations of these classes. In particular, we show that some known families of graph classes (e.g. subfactorial classes and classes of structurally bounded degree) can be characterized precisely as Boolean functions of some classes of equivalence graphs.

These results demonstrate the usefulness of Boolean functions of graph classes in studying the poset of hereditary graph classes under containment. This is a complex structure, and one way to tame its complexity is by coarsening it using reductions between classes and considering only representative classes (see \cite{braunfeld2025first} for an example of such a coarsening via first-order transductions between graph classes). Reductions between graph classes via Boolean functions provide another way to coarsen the poset, and our characterization of Boolean combinations of classes of equivalence graphs shows that this coarsening can naturally highlight interesting properties of graph classes, including subexponential speed, subfactorial speed, and structurally bounded degree.

\paragraph{Boolean functions of graphs and $\chi$-boundedness.}

Informally, a graph class is $\chi$-bounded, if  the chromatic number of every induced subgraph $G$ of every graph in the class can be bounded in terms of the clique number of $G$ via a function that depends only on the class.
Classes of graphs that are $\chi$-bounded are natural extensions of the class of perfect graphs,
which are exactly the graphs in which the chromatic number of each induced subgraph is equal to its clique number. 
A systematic study of $\chi$-bounded graph classes was initiated by Gy{\'a}rf{\'a}s \cite{Gya87} in the 1980s and remains an active area of research to date (see the survey \cite{SS20}).
The main goal of this line of research is to understand which graph classes are $\chi$-bounded.
Since Boolean functions allow us to construct new graph classes from a given class in a controlled but rather general way, it is natural to understand how Boolean combinations affect $\chi$-boundedness. 
We investigate this question in detail in \cref{sec:chi-boundedness}.

In general, it is known that Boolean combinations do not preserve $\chi$-boundedness (see \cref{sec:intbdd}).
We identify in \cref{sec:intbdd} some broad conditions on a class that imply preservation of $\chi$-boundedness under Boolean combinations.
We further show that $\chi$-boundedness is preserved in some special classes of graphs including split graphs (\cref{sec:interval-graphs}) and permutation graphs (\cref{sec:permutation}). In fact, for the latter two classes, we show that their Boolean combinations are \emph{polynomially} $\chi$-bounded, i.e., the chromatic number is upper-bounded by a polynomial function of the clique number.
Finally, in \cref{sec:chi-boundedness-equiv-graphs} we study $\chi$-boundedness of Boolean combinations of classes of equivalence graphs, one of the simplest classes of perfect graphs.
Equivalence graphs are permutation graphs, and therefore, by the above, their Boolean combinations are polynomially $\chi$-bounded. We show that, in general, the polynomial upper bound cannot be improved (\cref{sec:chi-non-linear}). We complement this by identifying some restrictions under which Boolean combinations of equivalence graphs are \emph{linearly} $\chi$-bounded (\cref{sec:chi-linear}), and even perfect (\cref{sec:chi-perfect}).

\section{Preliminaries}

In this section, we introduce most of the notions and notation that we use in the paper. 
We denote by $\bN$ the set of natural numbers, i.e., the set of positive integers. For a set $S$ and a natural number $k \in \bN$, we denote by $\binom{S}{k}$ the set of all $k$-element subsets of $S$.

\subsection{Graphs}

All graphs in this work are simple, i.e., undirected, without loops or multiple edges. 
We denote by $V(G)$ and $E(G)$ the vertex set and the edge set of a graph $G$, respectively. For distinct $a,b \in V(G)$, we denote by $G(a,b)$ the adjacency indicator of $a$ and $b$, i.e., $G(a,b) = 1$ if $(a,b) \in E(G)$ and $G(a,b) = 0$ otherwise. 
Given a graph $G=(V,E)$ and a vertex $v \in V$, the neighborhood of $v$ in $G$ is the set of its neighbors, i.e., vertices $w\in V$, such that $w$ is adjacent to $v$. The neighborhood of $v$ is denoted by $N_G(v)$. The degree of $v$, denoted by $\degree_G(v)$, is the cardinality of $N_G(v)$; the co-degree of $v$ is $|V|-\degree_G(v)-1$, i.e., the number of non-neighbors of $v$. 
When $G$ is clear from the context, we sometimes write $N(v)$ and $\degree(v)$ instead of $N_G(v)$ and $\degree_G(v)$.
The \emph{maximum degree of $G$}, denoted $\Delta(G)$, is the maximum degree of a vertex in $G$.
A \emph{complete} graph on a vertex set $V$ is the graph in which all vertices are pairwise adjacent; an \emph{empty} graph on $V$ is the graph with no edges. 
A graph is \emph{homogeneous} if it is either a complete graph or an empty graph.
A set of pairwise adjacent vertices in $G$ is a \emph{clique} of $G$ and a set of pairwise non-adjacent vertices is an \emph{independent set} of $G$.
The maximum size of a clique (respectively, an independent set) in $G$ is called the \emph{clique number} (respectively, the \emph{independence number}) of $G$ and denoted by $\omega(G)$ (respectively, $\alpha(G)$).
A graph is \emph{bipartite} if its vertex set can be partitioned into two independent sets, called \emph{parts} of the bipartite graph.
A bipartite graph is \emph{complete} if every vertex in one of its parts is adjacent to every vertex in the other part.
A \emph{biclique} of size $k$ is a complete bipartite graph with $k$ vertices in each of the parts.

For a set of vertices $S \subseteq V$, we denote by $G[S]$ the subgraph of $G$ \emph{induced} by $S$, i.e., the graph with vertex set $S$ where $x,y \in S$ are adjacent in $G[S]$ if and only if $x,y$ are adjacent in $G$.
For disjoint sets $A,B \subseteq V$, we denote by $G[A,B]$ the bipartite graph with parts $A,B$, where $a \in A$ and $b \in B$ are adjacent in $G[A,B]$ if and only if they are adjacent in $G$; we say that $A$ is \emph{complete to} $B$ in $G$ if $G[A,B]$ is complete, and we say that $A$ is \emph{anticomplete to} $B$ in $G$ if $G[A,B]$ has no edges.
The \emph{complement} $\overline{G}$ of $G$ is the graph with the same vertex set as $G$ in which two vertices are adjacent if and only if they are not adjacent in $G$.
 The \emph{chromatic number} $\chi(G)$ of $G$ is the minimum number of colors in a proper vertex coloring of $G$, i.e., a coloring of its vertices in which no two adjacent vertices get assigned the same color.

Given two graphs $G_1=(V_1,E_1)$ and $G_2=(V_2,E_2)$, the \emph{union} $G_1 \cup G_2$ and the \emph{intersection} $G_1 \cap G_2$ of the two graphs are the graphs $(V_1 \cup V_2, E_1 \cup E_2)$ and $(V_1 \cap V_2, E_1 \cap E_2)$, respectively.
If $V_1$ and $V_2$ are disjoint, then we call $G_1 \cup G_2$ the disjoint union of $G_1$ and $G_2$ and denote it by $G_1+G_2$. 
For a graph $G$ and $p \in \bN$ we denote by $pG$ a graph isomorphic to the disjoint union of $p$ copies of $G$.

We use standard notation for common graphs: $K_n$, $O_n$, $K_{n,m}$, $C_n$, $P_n$ denote respectively the $n$-vertex complete graph, the $n$-vertex empty graph, the complete bipartite graph with the parts of size $n$ and $m$, the $n$-vertex cycle, and the $n$-vertex path.

An \emph{equivalence graph} is a graph in which every connected component is a complete graph.

\paragraph{Graph parameters.}
We define a \emph{graph parameter} as a function $\sigma$ that assigns to every graph a non-negative integer such that
$\sigma(G_1) = \sigma(G_2)$ whenever $G_1$ and $G_2$ are isomorphic. The graph parameter $\sigma$ is bounded in a class of graphs $\cX$ if
there exists $c \in \bN$ such that $\sigma(G) \leq c$ for every $G \in \cX$; otherwise $\sigma$ is unbounded in $\cX$.

The maximum degree, the clique number, the independence number, and the chromatic number are examples of graph parameters. We now define some further graph parameters that we consider in this paper. For a graph $G$,
\begin{enumerate}
    \item the \emph{degeneracy} of $G$, denoted $\dgn(G)$, is the minimum integer $d$
    such that every subgraph of $G$ has a vertex of degree at most $d$; 

    \item the \emph{biclique number} of $G$ is the maximum $k\in \bN$ such that there exist two disjoint sets $A = \{a_1, \ldots, a_k\}$ and $B = \{b_1, \ldots, b_k\}$ where $a_i$ is adjacent to $b_j$ in $G$ for all~$i,j \in [k]$;

    \item the \emph{chain number} of
    $G$, denoted $\ch(G)$, is the maximum $k\in \bN$ such that there exist two disjoint sets $A = \{a_1, \ldots, a_k\}$ and $B = \{b_1, \ldots, b_k\}$ where $a_i$ is adjacent to $b_j$ in $G$ if and only if $i \leq j$.
\end{enumerate}

\paragraph{Complementation operations.}
We will use a number of known complementation operations and also introduce a new complementation operation that generalizes the known ones. As we shall observe, each of these operations can be seen as the XOR function of the input graph with a certain equivalence graph.

Let $G=(V,E)$ be an $n$-vertex graph.
Given a vertex $v$ of $G$ \emph{the local complementation of $G$ centered at $v$} is the graph obtained from $G$ by complementing the subgraph induced by $N(v)$, the neighborhood of $v$.
A graph $H$ is a local complementation of $G$, if $H$ is the local complementation of $G$ centered at $v$ for some $v \in V$.
The local complementation operation is known to preserve the rank-width of a graph \cite{Oum05}.

The more general operation of subgraph complementation was studied in the context of clique-width \cite{KLM09}. In particular, it was shown that this operation preserves boundedness of clique-width.
Given a subset $U \subseteq V$, the \emph{subgraph complementation of $G$ with respect to $U$} is the graph that is obtained from $G$ by complementing the edges in the induced subgraph $G[U]$. A graph $H$ is a subgraph complementation of $G$, if $H$ is the subgraph complementation of $G$ with respect to some subset of $V$.
We note that subgraph complementation generalizes both the standard graph complementation and the local complementation operations (the local complementation of $G$ centered at some vertex $v$ is the subgraph complementation of $G$ with respect to $N(v)$).

We now introduce a new complementation operation that generalizes the subgraph complementation operation.
Given a partition $\cP = (V_1, V_2, \ldots, V_k)$ of $V$, the \emph{partition complementation of $G$ with respect to $\mathcal{P}$}
is the graph that is obtained from $G$ by complementing the edges in each of the induced subgraphs $G[V_1], G[V_2], \ldots, G[V_k]$.
A graph $H$ is a partition complementation of $G$, if $H$ is the partition complementation of $G$ with respect to some partition of $V$.
We observe that the subgraph complementation of $G$ with respect to $U$ is the partition complementation of $G$ with respect to $\cP$, where the partition classes of $\cP$ include $U$ and one singleton set for each vertex in~$V \setminus U$.

\paragraph{Graph classes.}
A \emph{class} of graphs is a set of graphs that is closed under isomorphism. Throughout the paper, we assume that every class contains at least one graph on $n$ vertices for each $n \in \bN$. For a class of graphs $\cX$, we denote by $\cX^n$ the set of \emph{labeled} $n$-vertex graphs in $\cX$, i.e., graphs in $\cX$ with vertex set $[n]:= \{ 1,2, \ldots, n \}$.
The class $\overline{\cX}$ consists of the graphs $G$ such that $\overline{G}\in\cX$. Class $\cX$ is called \emph{self-complementary} if $\cX = \overline{\cX}$.

A graph class $\cX$ is \emph{hereditary} if it is closed under taking induced subgraphs or, equivalently, is closed under removing vertices. Class $\cX$ is \emph{monotone} if it is closed under taking subgraphs or, equivalently, is closed under removing edges and vertices. Class $\cX$ is \emph{intersection-closed} if for any two $G_1,G_2 \in \cX$, not necessarily on the same vertex set, the intersection $G_1 \cap G_2$ also belongs to $\cX$. Every monotone class is intersection-closed as the intersection of any two graphs is a subgraph of each of them. Furthermore, every intersection-closed class is hereditary. Indeed, let $\cX$ be an intersection-closed class, $G=(V,E)$ be a graph in $\cX$, and $v$ be a vertex in $G$. We claim that the graph $H$ obtained from $G$ by removing $v$ belongs to $\cX$. To see this, define $G'=(V',E')$ as a graph obtained from $G$ by renaming $v$; more formally, let $v'$ be an element outside $V$ and define $V' = (V \cup \{v'\}) \setminus \{v\}$ and $E'=\{ (x,y) \in E : x \neq v, y \neq v\} \cup \{(v',y) : (v,y) \in E \}$. Then $G'=(V',E')$ is isomorphic to $G$ and thus belongs to $\cX$. Furthermore, $G \cap G'$ is equal to $H$ and hence $H$ belongs to $\cX$.
Given a hereditary graph class $\cX$, we say that a graph class $\cY$ is a \emph{proper} subclass of $\cX$ if $\cY$ is hereditary and strictly contained in $\cX$. 

Let $M$ be a set of graphs. We say that a graph $G$ is \emph{$M$-free} if no member of $M$ occurs as an induced subgraph of $G$. 
If $M$ consists of a single graph $H$, we write $H$-free instead of $\{H\}$-free. 
It is a well-known fact that a graph class $\cX$ is hereditary if and only if there exists a set $M$ of graphs such that a graph $G$ is in $\cX$ if and only if $G$ is $M$-free. In this case, set $M$ is called a set of \emph{forbidden induced subgraphs} for class $\cX$. 

\paragraph{Properties of graph classes.}
Let $\cX$ be a class of graphs. The function $n \mapsto |\cX^n|$ is the \emph{speed} of $\cX$.

The class $\cX$ is said to have the \emph{\EH property} if there exists a constant $\delta > 0$ such that every $n$-vertex graph $G \in \cX$ has an independent set or a clique of size at least $n^{\delta}$, i.e.,
if $\max\{ \alpha(G), \omega(G) \} \geq n^{\delta}$.

A function $f : \bN \rightarrow \bN$ is a \emph{$\chi$-binding function} for a graph class $\cX$ if $\chi(H)\leq f(\omega(H))$ holds for every induced subgraph $H$ of a graph $G \in \cX$; we shall always assume that $f(x) \geq x$ for every $x \in \bN$.
A graph class $\cX$ is \emph{$\chi$-bounded} if there exists a $\chi$-binding function for $\cX$.
If $\cX$ is $\chi$-bounded, then it has a \emph{smallest $\chi$-binding function} defined as
\[
f^\ast(k)=
\begin{cases}
\displaystyle
\max\bigl\{\chi(H) :  H \text{ induced subgraph of } G \in \cX,\ \omega(H)=k\bigr\},
& \text{if such a graph \(H\) exists},\\[1ex]
k,
& \text{otherwise}.
\end{cases}
\]

Every graph parameter can be used to define a graph class property as ``boundedness'' of the parameter.
For example, a graph class is said to be
\begin{enumerate}
    \item of \emph{bounded degree} if the maximum degree is bounded in the class; 
    
    \item of \emph{bounded degeneracy} if the degeneracy is bounded in the class;

    \item \emph{weakly-sparse} if the biclique number is bounded in the class; and
    
    \item \emph{edge-stable} if the chain number is bounded in the class.
\end{enumerate}

\subsection{Boolean functions}

We usually represent Boolean functions with Boolean formulae which consist of variables and logical connectives (such as NOT, AND, OR, XOR). 
We say that a Boolean formula is in disjunctive normal form (DNF) if it is a disjunction consisting of one or more conjunctive clauses, each of which is a conjunction of one or more literals (variables or their negations). 

A Boolean function $f: \zo^k \rightarrow \zo$ is called \emph{monotone} if
$f(a_1, a_2, \ldots, a_k) \leq f(b_1, b_2, \ldots, b_k)$ for every 
$(a_1, a_2, \ldots, a_k), (b_1, b_2, \ldots, b_k) \in \zo^k$ with $a_i \leq b_i$ for every $i \in [k]$. Every monotone Boolean function can be represented with a \emph{monotone} DNF, i.e., a DNF in which all literals are non-negated.

\begin{theorem}[see e.g.~{\cite[Theorem 1.21]{CH11}}]\label{th:monotone-DNF}
    If $f$ is a monotone Boolean function, then $f$ can be represented with a monotone DNF.
\end{theorem}

Furthermore, it is known that any Boolean function can be uniquely represented as an XOR of conjunctions of positive (i.e.~non-negated) variables.

\begin{theorem}[Zhegalkin polynomial aka Algebraic Normal Form \cite{Zhe27}]\label{th:anf}
    Any Boolean function $f(x_1, x_2, \ldots, x_k)$ has a unique representation in Algebraic Normal Form, i.e., there exist 
    distinct sets $I_1, I_2, \ldots, I_s \subseteq [k]$ such that
    $$
        f(x_1, x_2, \ldots, x_k) = \Xor_{j=1}^{s} \bigwedge_{i \in I_j} x_i.
    $$
\end{theorem}

\noindent
We note that in \cref{th:anf}, it is possible that $I_j = \emptyset$ for some $j \in [s]$. In such a case, $\bigwedge_{i\in I_j} x_i=1$.

\subsection{Boolean functions of graphs}

\paragraph{Functions of graphs.}
Let $G, H_1, H_2, \ldots, H_k$ be $n$-vertex graphs on the same vertex set $V$. We say that $G$ is a \emph{Boolean combination} or a \emph{Boolean function} (or simply a \emph{function}) of $H_1, H_2, \ldots, H_k$, if there exists a Boolean function
$f : \{0,1\}^k \rightarrow \{0,1\}$ such that 
$$
    G(a,b) = f\left( H_1(a,b), H_2(a,b), \ldots, H_k(a,b) \right)
$$
holds for all distinct $a,b \in V$. 
In other words, the adjacency matrix of $G$ is obtained by applying $f$ to the adjacency matrices of $H_1, H_2, \ldots, H_k$ entry-wise (except the elements on the main diagonal).
Abusing notation, we will sometimes write $G = f(H_1, H_2, \ldots, H_k)$.
We note that if $G$ is a function of $k$ graphs, then it is also a function of $k'$ graphs for any $k' \geq k$. Indeed, $f$ can be viewed as a $k'$-variable function by introducing $k'-k$ inessential variables, whose values do not affect the output. These variables may therefore be assigned arbitrary graphs; for example, setting each of them equal to $H_1$ gives
$G=f(H_1,\ldots,H_k,H_1,\ldots,H_1)$.
If $f$ is defined by $x_1 \vee x_2 \vee \cdots \vee x_k$, $x_1 \wedge x_2 \wedge \cdots \wedge x_k$, or $x_1 \xor x_2 \xor \cdots \xor x_k$, we will say that $G$ is the $k$-union, $k$-intersection, $k$-XOR, respectively, of $H_1, H_2, \ldots, H_k$. In these cases, we may write $G = H_1 \vee H_2 \vee \cdots \vee H_k$, $G = H_1 \wedge H_2 \wedge \cdots \wedge H_k$, $G = H_1 \xor H_2 \xor \cdots \xor H_k$, respectively.
In particular, the notation $H_1 \vee H_2$, $H_1 \wedge H_2$, $H_1 \xor H_2$
assumes that both $H_1=(V,E_1)$ and $H_2=(V,E_2)$ have the same vertex set, and   denotes the graphs $(V, E_1 \cup E_2)$, $(V, E_1 \cap E_2)$, and $(V, (E_1 \setminus E_2) \cup (E_2 \setminus E_1))$, respectively.

We make some useful observations about Boolean combinations of graphs involving XOR functions. In particular, the first three observations describe how graph complementation, subgraph complementation, and partition complementation can be seen as XOR of the graph with certain equivalence graphs.

\begin{observation}\label{obs:xor-with-complete-graph}
    Let $G$ be an $n$-vertex graph. Then $\overline{G} = G \xor K_n$.
\end{observation}

\begin{observation}\label{obs:xor-with-clique}
    Let $G=(V,E)$ be an $n$-vertex graph and let $S \subseteq V$ be a vertex subset of $G$. Let $H=(V,E')$ be the graph in which $S$ is a clique and all other vertices are isolated. 
    Then the subgraph complementation of $G$ with respect to $S$ is equal to $G \xor H$.
\end{observation}

\begin{observation}\label{obs:xor-with-equivalence}
    Let $G=(V,E)$ be a graph and $\mathcal{P}=(V_1,V_2,\ldots,V_k)$ be a partition of $V$. Let $H=(V,E')$ be the equivalence graph in which every $V_i$ is a maximal clique.
    Then the partition complementation of $G$ with respect to $\mathcal{P}$ is equal to $G \xor H$.
\end{observation}

\begin{observation}\label{obs:xor-of-k-graphs}
    Let $G = H_1 \xor H_2 \xor \cdots \xor H_k$. Then two vertices $a,b \in V(G)$ are adjacent in $G$ if and only if $a$ and $b$ are adjacent in an odd number of graphs $H_1, H_2, \ldots, H_k$,
    i.e.,
    \[
        (a,b) \in E(G) \quad \Longleftrightarrow \quad 
        | \{i ~:~ (a,b) \in E(H_i), i \in [k] \}| \equiv 1 \pmod{2}.
    \]
\end{observation}

\paragraph{Functions of graph classes.}
Let $\cX$ and $\cY$ be two classes of graphs and $k \in \bN$. We say that the class $\cY$ is a \emph{$k$-function} of $\cX$ if every graph $G \in \cY$ is a function of $k$ graphs in $\cX$ (note that the function $f$ such that $G=f(H_1,\dots, H_k)$ for some $H_1, \dots, H_k\in \cX$ may depend on $G$). If $\cY$ is a $k$-function of $\cX$ for some $k \in \bN$, then we say that $\cY$ is a \emph{function} of $\cX$.
For graph classes $\cX_1, \cX_2, \ldots, \cX_k$ and a Boolean function
$f : \{0,1\}^k \rightarrow \{0,1\}$, we write $f(\cX_1, \cX_2, \ldots, \cX_k)$ to denote the class of graphs $\{G=f(H_1, H_2, \ldots, H_k) : H_i=(V,E_i) \in \cX_i \}$.
It is easy to see that if all $\cX_i$, $i \in [k]$, are hereditary
then $f(\cX_1, \cX_2, \ldots, \cX_k)$ is also hereditary. 
If all $\cX_i$, $i \in [k]$, are the same and equal to $\cX$, we will sometimes write $f(\cX)$ to denote $f(\cX_1, \cX_2, \ldots, \cX_k)$.

\begin{remark}\label{rem:single-function}
    We note, by definition, if a class $\cY$ is a $k$-function of a class $\cX$, then $\cY \subseteq \bigcup_{f} f(\cX)$, where the union is taken over all $k$-variable Boolean functions. For some class properties, such as (linear/polynomial) $\chi$-boundedness, a set-theoretic union of boundedly many classes each of which possesses the property, also has the property. In such cases, we will usually establish the property for $f(\cX)$ for an arbitrary, but fixed $k$-variable Boolean function $f$.
\end{remark}

We denote 
\[
    \bigvee_{i=1}^k \cX_i := \big\{ H_1 \vee H_2 \vee \cdots \vee H_k : H_i=(V,E_i) \in \cX_i, i \in [k] \big\},
\]
\[
    \bigwedge_{i=1}^k \cX_i := \big\{ H_1 \wedge H_2 \wedge \cdots \wedge H_k : H_i=(V,E_i) \in \cX_i, i \in [k] \big\},
\]
\[
    \Xor_{i=1}^k \cX_i := \big\{ H_1 \xor H_2 \xor \cdots \xor H_k : H_i=(V,E_i) \in \cX_i, i \in [k] \big\},
\]
and call these classes the union, intersection, and XOR of $\cX_1, \cX_2, \ldots, \cX_k$, respectively.
If all $\cX_i$, $i \in [k]$, are the same and equal to $\cX$, then we call $\bigvee_{i=1}^k \cX_i$, $\bigwedge_{i=1}^k \cX_i$, and $\Xor_{i=1}^k \cX_i$, the $k$-union, $k$-intersection, and $k$-XOR of $\cX$, respectively.

It is easy to see that if $\cX_i$, $i \in [k]$, are monotone,
then $\bigvee_{i=1}^k \cX_i$, $\bigwedge_{i=1}^k \cX_i$, and $\Xor_{i=1}^k \cX_i$ are all monotone.
However, if $\cX$ is monotone, $\overline{\cX}$ is not necessarily monotone (which can be easily seen with the class of all edgeless graphs).

\section{Boolean functions and properties of graph classes}
\label{sec:properties}

In this section, we show how Boolean combinations affect various properties of graph classes.

Let $\cX$ and $\cY$ be two graph classes such that $\cY$ is a function of $\cX$.
We observe that the number of graphs in $\cY$ cannot be too much larger than that in $\cX$ (\cref{sec:speed}), and 
note that an adjacency labeling scheme for graphs in $\cX$ can naturally be used to construct an adjacency labeling scheme for graphs in $\cY$ of the same size up to a constant factor (\cref{sec:adj-labeling}).
We further show that the \EH property (\cref{sec:EH-property}) and edge-stability (\cref{sec:stability}) are preserved by Boolean combinations of graph classes. We also analyze how Boolean combinations affect the 
neighborhood complexity and VC dimension of a graph class (\cref{sec:neighborhood-complexity-vc-dimension}).

\subsection{Speed of graph classes}
\label{sec:speed}

Recall, the speed of a graph class $\cX$ is the function $n \mapsto |\cX^n|$, where 
$\cX^n$ is the set of graphs in $\cX$ with vertex set $[n]$.

\begin{lemma}\label{lem:speed}
    Let $\cX$ and $\cY$ be two classes of graphs, such that 
    $\cY$ is a $k$-function of $\cX$. Then $| \cY^n | \leq 2^{2^k} | \cX^n |^k$.
\end{lemma}
\begin{proof}
    Let $k \in \bN$ be such that $\cY$ is a $k$-function of $\cX$.
    By definition, any $n$-vertex graph in $\cY$ is determined 
    by an ordered $k$-tuple $(H_1, H_2, \ldots, H_k)$ of $n$-vertex graphs from $\cX$ and a Boolean function $f : \zo^k \rightarrow \zo$.
    Thus, the number of graphs in $\cY^n$ does not exceed $2^{2^k} |\cX^n|^k$.
\end{proof}

\subsection{Adjacency labeling schemes}
\label{sec:adj-labeling}

Let $\cX$ be a~class of graphs and $b : \bN \rightarrow \bN$ be a~function. An \emph{$b(n)$-bit adjacency labeling scheme} for $\cX$ is a~pair (encoder, decoder) of algorithms where for any $n$-vertex graph $G\in \cX^n$ the encoder assigns binary strings, called \emph{labels}, of length $b(n)$ to the vertices of $G$ such that the adjacency between any pair of vertices can be inferred by the decoder only from their labels.
We note that the decoder depends on the class $\cX$, but not on the graph~$G$.
The function~$b(\cdot)$ is the \emph{size} of the labeling scheme, see the book~\cite{Spinrad03} and more recent work (e.g.~\cite{Alo24,BDSZ24}) for further background.

\begin{lemma}\label{lem:LS}
    Let $\cX$ and $\cY$ be two classes of graphs, such that 
    $\cY$ is a function of $\cX$.
    If $\cX$ admits a $b(n)$-bit adjacency labeling scheme, then
    $\cY$ admits an $O(b(n))$-bit adjacency labeling scheme.
\end{lemma}
\begin{proof}
    Let $k \in \bN$ be such that $\cY$ is a $k$-function of $\cX$,
    and let $G=(V,E)$ be an arbitrary $n$-vertex graph in $\cY$.
    By assumption there exists a Boolean function $f : \zo^k \rightarrow \zo$ and $k$ graphs $H_1, H_2, \ldots, H_{k} \in \cX$ on the vertex set $V$ such that for any distinct $a,b \in V$ we have 
    $$
        G(a,b) = f(H_1(a,b), H_2(a,b), \ldots, H_k(a,b)).
    $$ 
    
    We describe the encoder for graph $G$, and the corresponding decoder for the class $\cY$.
    
    \noindent
    \textbf{Encoder.} Let $\ell_1, \ell_2, \ldots, \ell_k : [n] \rightarrow \{0,1\}^{b(n)}$ be functions that assign adjacency labels of size $b(n)$ to the vertices of the graphs $H_1, H_2, \ldots, H_k$, respectively.
    We define the labeling function $\ell : [n] \rightarrow \{0,1\}^{b'(n)}$, with $b'(n) =  k+1 + k \cdot b(n) + 2^k =O(b(n))$ as follows. For a vertex $a$ of $G$, the label $\ell(a)$ is the concatenation of the unary representation of $k$, the labels $\ell_1(a), \ell_2(a), \ldots, \ell_k(a)$, and the vector $\text{val}(f)$ of the values of $f$ on all $2^k$ input binary vectors.

    \medskip
	
    \noindent
    \textbf{Decoder.} Given two vertices $a$ and $b$, their adjacency in $G$ can be decided by (1) evaluating the adjacencies of these vertices in each $H_i$, $i \in [k]$, using the components $\ell_i(a)$ and $\ell_i(b)$ of their labels $\ell(a)$ and $\ell(b)$, respectively, and the decoder for class $\cX$; (2) reading the value of $f$ on the arguments $H_1(a,b), H_2(a,b), \ldots, H_k(a,b)$ from the vector $\text{val}(f)$.
\end{proof}

\subsection{\texorpdfstring{Erd\H{o}s}{Erdos}-Hajnal property}
\label{sec:EH-property}

\begin{lemma}\label{lem:sEH-general}
	Let $\cX$ be a hereditary class of graphs and $\cY$ be a class of graphs.
	If $\cY$ is a function of $\cX$, and $\cX$ has the \EH property, then $\cY$ does too.
\end{lemma}
\begin{proof}
	Let $k$ be a constant such that every graph in $\cY$ is a $k$-function of graphs in $\cX$. Let $G$ be an arbitrary $n$-vertex graph and let $f : \{0,1\}^k \rightarrow \{0,1\}$ and $H_1, H_2, \ldots, H_k \in \cX$, be such that $G = f(H_1, H_2, \ldots, H_k)$.
	
	Assume $\cX$ has the \EH property, i.e., there exists a constant $\delta \in (0,1]$, such that every $n$-vertex graph in $\cX$ has a homogeneous subgraph of size at least $n^{\delta}$.
 	We claim that $G$ has a homogeneous subgraph of size at least $n^{\delta^k}$, and thus $\cY$ has the \EH property. To prove this, it is enough to show that there exists a set $S \subseteq [n]$ with at least $n^{\delta^k}$ elements such that each $H_i[S]$, $i \in [k]$, is homogeneous. Indeed, then for all pairs $a,b \in S$, $a \neq b$, the values $f(H_1(a,b), H_2(a,b), \ldots, H_k(a,b))$ are the same, because the tuples of arguments are the same.
	
	We prove the existence of $S$ by induction on $k$. If $k=1$, then, since $H_1$ belongs to $\cX$, there is a set $S_1$ of size at least $n^{\delta}$ such that $H_1[S_1]$ is homogeneous. In this case we take $S$ to be $S_1$. Suppose now there exists a set $S_{k-1}$ of size at least $n^{\delta^{(k-1)}}$ such that each $H_i[S_{k-1}]$, $i \in [k-1]$ is homogeneous. Since $\cX$ is hereditary, $H_k[S_{k-1}]$ belongs to $\cX$ and thus there exists a set $S_k \subseteq S_{k-1}$ of size $|S_k| \geq |S_{k-1}|^{\delta} \geq (n^{\delta^{(k-1)}})^{\delta} = n^{\delta^k}$ such that $H_k[S_k]$ is homogeneous.
	Consequently, all $H_i[S_k]$, $i \in [k]$ are homogeneous, and we take $S$ to be $S_k$.
\end{proof}

\subsection{Edge-stability}
\label{sec:stability}

In this section we show that Boolean combinations of a graph class preserve edge-stability. This is a special case of a known fact in model theory (see e.g.~\cite[Remark 3.4]{palacin2018introduction}) and learning theory (see e.g.~\cite{AlonBMS20}). For completeness, we state it here and present a proof in the graph-theoretic language. 

We start by introducing the notion of strong chain number, which is closely related to the chain number, but more convenient to work with.
The \emph{strong chain number} of a graph $G$, denoted $\sch(G)$,  is the maximum $k\in \bN$ such that there exist two disjoint sets $A = \{a_1, \ldots, a_k\}$ and $B = \{b_1, \ldots, b_k\}$ where $a_i$ is adjacent to $b_j$ in $G$ if $i < j$, and $a_i$ is not adjacent to $b_j$ in $G$ if $i > j$.

Note that the strong chain number is similar to the chain number, but we are not concerned with the adjacencies between $a_i$ and $b_i$.
The following lemma (adapted from \cite{pilipczuk2022transducing}) relates the chain number $\ch(G)$ and the strong chain number $\sch(G)$, of a graph $G$.

\begin{lemma}\label{lem:2chains}
    For a graph $G$, we have that $\lfloor \sch(G)/2 \rfloor\leq \ch(G)\leq \sch(G)$.
\end{lemma}
\begin{proof}
    Clearly, by the definitions of $\ch(G)$ and $\sch(G)$, we have $\ch(G)\leq \sch(G)$.
    Now, let $\sch(G)=k$. Then $G$ contains two disjoint sets $A = \{a_1, \ldots, a_k\}$ and $B = \{b_1, \ldots, b_k\}$ where $a_i$ is adjacent to $b_j$ and $a_j$ is non-adjacent to $b_i$ in $G$ for all $1\leq i < j \leq k$. If we set $c_i=a_{2i-1}$ and $d_i=b_{2i}$ for $1\leq i\leq \lfloor k/2 \rfloor$, then $\{ c_1,\ldots,c_{\lfloor k/2\rfloor} \}$ and $\{ d_1,\ldots,d_{\lfloor k/2\rfloor} \}$ are two disjoint sets such that $c_i$ is adjacent to $d_j$ if and only if $i \leq j$. Consequently, $\ch(G)\geq \lfloor k/2 \rfloor$.
\end{proof}

The following result with a worse bound was originally proved in \cite{AlonBMS20}, in a learning theory setting. The bound stated here was obtained in \cite{GhaziG0M21}, and we adapt their proof to a graph theory setting. 
\begin{lemma}[\cite{GhaziG0M21}]\label{lem:stability}
    Let $s \in \bN$ and $\cX$ be a class of graphs with chain number less than $s$. 
    If $\cY$ is a $k$-function of $\cX$,  then $\cY$ has chain number less than $(2k)^{4ks}$.
\end{lemma}
\begin{proof}
    Let $G=(V,E)\in\cY$ and $G=f(H_1,\ldots,H_k)$ for some $H_1,\ldots,H_k\in \cX$ and a $k$-variable Boolean function $f$. 
    Towards a contradiction, suppose that $G$ has chain number $t \geq (2k)^{4ks}$. Thus, there exist two disjoint sets of vertices, $A=\{a_1,\ldots,a_t\}$ and $B=\{b_1,\ldots,b_t\}$ in $V$, such that $a_i$ is adjacent to $b_j$ in $G$ if and only if $i \leq j$. 

    Note that for every $1 \leq i < j \leq t$, there must exist $\ell \in [k]$ such that $H_{\ell}(a_i,b_j)\neq H_{\ell}(a_j,b_i)$, as otherwise we would have
    \[              1=G(a_i,b_j)=f(H_1(a_i,b_j),\ldots,H_k(a_i,b_j))=f(H_1(a_j,b_i),\ldots,H_k(a_j,b_i))=G(a_j,b_i)=0,
    \]
    which is not possible.
    
    Define a complete graph $G'$ with vertex set $[t]$ and edges colored as follows.
    For every $1 \leq i < j \leq t$ and the smallest $\ell \in [k]$ such that $H_{\ell}(a_i,b_j)\neq H_{\ell}(a_j,b_i)$,
    color the edge between $i,j$ by $(\ell,0)$ if $H_{\ell}(a_i,b_j)=0$, and
    by $(\ell,1)$ if $H_{\ell}(a_i,b_j)=1$.
    Clearly, such a coloring uses at most $2k$ colors.

    It follows from \cite{greenwood1955combinatorial}, that if $n\geq c^{rc}$, then in every edge coloring of $K_n$ with $c$ colors there exists a monochromatic clique of size $r$. 
    Thus, $G'$ contains a monochromatic clique $C$ with $2s$ vertices. 
    Let $c_1,\ldots,c_{2s} \in [t]$ be the vertices of $C$, where $c_1 < c_2 < \cdots < c_{2s}$.

    If the edges of $G'[C]$ are colored $(\ell,0)$,
    then for each $1 \leq i < j \leq 2s$ we have that $H_{\ell}(a_{c_i}, b_{c_j}) = 0$ and $H_{\ell}(a_{c_j}, b_{c_i}) = 1$.
    If the edges of $G'[C]$ are colored $(\ell,1)$,
    then for each $1 \leq i < j \leq 2s$ we have that $H_{\ell}(b_{c_j}, a_{c_i}) = 1$ and $H_{\ell}(b_{c_i}, a_{c_j}) = 0$. Both cases witness $\sch(H_{\ell})\geq 2s$, and so, by Lemma \ref{lem:2chains}, $\ch(H_{\ell})\geq s$.
    The latter contradicts the assumption that the chain number of any graph in $\cX$ is less than $s$.
 \end{proof}

As a corollary we obtain the following.

\begin{theorem}
    If $\cY$ is a function of $\cX$ and $\cX$ is edge-stable, then $\cY$ is edge-stable.
\end{theorem}

\subsection{Neighborhood complexity and VC dimension}
\label{sec:neighborhood-complexity-vc-dimension}

A common approach to capture the complexity of a graph is by using complexity measures of the set system induced by the neighborhoods of the vertices of the graph. Two such measures that we will consider here are the neighborhood complexity and the VC dimension of a graph. They are respectively defined as the shatter function and the VC dimension of the corresponding set system. 
Graph classes with low neighborhood complexity have nice structural \cite{RVS19,EGKKP16,BDSZ24}, algorithmic \cite{EGKKP16,BKRTW22,DEMMPT23}, and combinatorial \cite{BDSZ24} properties. 
Similarly, graph classes of bounded VC dimension have useful combinatorial and structural properties \cite{Alo24,NSS23,JP24,FPS21}. 
We will show that Boolean combinations do not increase neighborhood complexity by much, and they also preserve boundedness of VC dimension. 
We proceed with the formal definitions.

A pair $(X, \cS)$, where $X$ is a~finite set and $\cS$ is a~family of subsets of $X$, is called a~set system.
The \emph{neighborhood set system} of a graph $G$ is the set system  $(V(G), \mathcal{N}_G)$ where $\mathcal{N}_G= \{ N_G(v) : v\in V(G)\}$.
Given a~set system $(X, \mathcal{S})$, a~subset $A \subseteq X$ is \emph{shattered} by $\cS$ if every subset $B$ of $A$ can be obtained as the intersection $B = A \cap Y$ for some $Y \in \cS$, i.e., $\{ Y \cap A~|~ Y \in \cS \} = 2^A$.
The \emph{Vapnik–Chervonenkis dimension} (or \emph{VC dimension} for short) of $\cS$ is the maximum size of a shattered subset of $X$.
The \emph{VC dimension of $G$} is the VC dimension of its neighborhood set system.
The VC dimension of a graph class $\cX$ is the minimum $d$ such that the VC dimension of no graph in $\cX$ exceeds~$d$.

The \emph{shatter function} of a~set system $(X, \mathcal{S})$ is the function $\pi_{\mathcal{S}}$ given by 
\[
    \pi_{\mathcal{S}}(m)= \max_{A\subseteq X, \, |A|=m} |\{Y\cap A~:~Y \in \mathcal{S} \}|. 
\]
The \emph{neighborhood complexity} of a~graph $G$, denoted by $\nu_{G}$, is the shatter function of its neighborhood set system, i.e., $\nu_{G}(m) = \pi_{\mathcal{N}_G}(m)$ for all $m \in [|V(G)|]$.
The \emph{neighborhood complexity} of a graph class $\cX$ is the function 
$\nu_{\cX}: \mathbb{N} \rightarrow \mathbb{N}$ defined by 
\[
    \nu_{\cX}(n) := \max_{G\in \cX,\, A\subseteq V(G), \,|A|=n} |\{N(v)\cap A~:~v \in V(G) \}| 
                            = \max_{G\in \cX} \pi_{\mathcal{N}_G}(n)
                            = \max_{G\in \cX} \nu_G(n).
\]

Next, we show how Boolean combinations affect the neighborhood complexity of a graph class.  

\begin{lemma}\label{lem:neighborhoodcomplexity}
    Let $k \in \bN$ and $\cX$ be a class of graphs.
    If $\cY$ is a $k$-function of $\cX$, then $\nu_{\cY}(n) \leq (\nu_{\cX}(n))^{k}$ holds for every $n \in \bN$.
\end{lemma}
\begin{proof}
    Let $G=(V,E)\in\cY$ and $G=f(H_1,\ldots,H_k)$ for some $H_1,\ldots,H_k\in \cX$ and a $k$-variable Boolean function $f$. 
    By assumption, for any $i \in [k]$, $m \in \bN$, and any set $S \subseteq V$ of size $m$, we have $|\{N_{H_i}(v) \cap S : v\in V\}| \leq \nu_{\cX}(m)$. We wish to show that $|\{N_{G}(v)\cap S : v\in V\}| \leq (\nu_{\cX}(m))^k$.
    
    Fix $v\in V$. Then note that for any $w\in S$, as $G(v,w)=f(H_1(v,w),\ldots,H_k(v,w))$, we see that $G(v,w)$ is entirely determined by the ordered tuple $(H_1(v,w),\ldots,H_k(v,w))$, which in turn
    is determined by $N_{H_1}(v)\cap S,\ldots, N_{H_k}(v)\cap S$. It follows that $N_G(v)\cap S=\{w\in S : G(v,w)=1\}$ is determined by
    the ordered tuple $S(v) := (N_{H_1}(v)\cap S,\ldots, N_{H_k}(v)\cap S)$. 
    Thus, $$|\{N_G(v)\cap S : v\in V\}|\leq |\{S(v) : v\in V\}|\leq \prod_{i=1}^k |\{N_{H_i}(v)\cap S : v\in V\}|\leq \prod_{i=1}^k \nu_{H_i}(m)\leq (\nu_{\cX}(m))^k.$$  
    
    Since $G$ and $S$ were chosen arbitrarily, we have
    $$
        \nu_{\cY}(m) = \max_{G \in \cY, S\subseteq V(G), |S|=m} |\{N_G(v)\cap S : v\in V\}|\leq (\nu_{\cX}(m))^k,
    $$
    as desired.
\end{proof}

\begin{lemma}\label{lem:VC-dimension}
    Let $\cX$ be a class of graphs and $\cY$ be a class of graphs.
    If $\cY$ is a function of $\cX$, and $\cX$ has bounded VC dimension, then $\cY$ has bounded VC dimension. 
\end{lemma}
\begin{proof}
Let $k \in \bN$ be such that $\cY$ is a $k$-function of $\cX$. By the Sauer-Shelah lemma \cite{Sau72,She72}, if $\cX$ has VC dimension bounded by $d$, then $\nu_{\cX}(m) \leq \sum_{i=0}^d \binom{m}{i} = O(m^d)$. Thus, by \cref{lem:neighborhoodcomplexity}, we have that $\nu_{\cY}(m) = O(m^{kd})$. So there exists $d' := d'(k,d)$ such that $\nu_{\cY}(m)<2^m$ holds for every $m\geq d'$. This implies that $\cY$ has VC dimension less than $d'$.  
\end{proof}

\section{Boolean functions of special graph classes}
\label{sec:special-graph-classes}

In this section, we provide characterizations of Boolean combinations of some special classes of graphs.
In \cref{sec:intersection-closed-classes}, we observe a convenient way of representing Boolean combinations of intersection-closed graph classes. We apply this to Boolean combinations of monotone graph classes in \cref{sec:monotone-classes}, and, in particular, characterize Boolean combinations of graph classes of bounded degree, bounded degeneracy, bounded chromatic number, and weakly-sparse graph classes.
In \cref{sec:equivalence-graphs}, we consider a hierarchy of subclasses of equivalence graphs and characterize Boolean combinations of these classes. In particular, we show that some known families of graph classes can be characterized precisely as Boolean functions of some classes of equivalence graphs.

\subsection{Intersection-closed graph classes}
\label{sec:intersection-closed-classes}

In this section, we show that functions of intersection-closed graph classes have a convenient representation via XOR functions or their negations.

\begin{lemma}\label{lem:intersection-closed-xor-representation}
    Let $\cX$ be an intersection-closed class of graphs, and $\cY$ be a $k$-function of $\cX$.
    Then for every graph $G \in \cY$ there exist $s$ graphs $H_1, H_2, \ldots, H_s \in \cX$, where $s \leq 2^k$, such that $G = \Xor_{i=1}^s H_i$ or $G = \overline{\Xor_{i=1}^s H_i}$.
    Furthermore, if $\cX$ contains all complete graphs, then $G$ can always be expressed as $\Xor_{i=1}^s H_i$.
\end{lemma}
\begin{proof}
    Let $G$ be a graph in $\cY$. By assumption, there exists a Boolean function $f : \zo^k \rightarrow \zo$ such that $G=f(F_1, F_2, \ldots, F_k)$ for some $F_1, F_2, \ldots, F_k \in \cX$.

    Let $I_1, I_2, \ldots, I_s \subseteq [k]$ be such that
    \[
        f(x_1, x_2, \ldots, x_k) = \Xor_{j=1}^{s} \bigwedge_{i \in I_j} x_i,
    \]
    is the Algebraic Normal Form representation of $f$, which exists and is unique by \cref{th:anf}.
    Then, we have $G = \Xor_{j=1}^{s} H_j$, where $H_j = \bigwedge_{i \in I_j} F_i = \bigcap_{i \in I_j} F_i$ if $I_j\not= \emptyset$ and $H_j=K_n$ for $n = |V(G)|$ if $I_j=\emptyset$. It follows that if $I_j\not=\emptyset$, then $H_j$ belongs to $\cX$ as $\cX$ is intersection-closed by assumption. Thus, if one of the sets $I_1, I_2, \ldots, I_s$, say $I_s$, is empty and $K_n \not\in \cX$, we have $G = K_n \xor \Xor_{i=1}^{s-1} H_i= \overline{\Xor_{i=1}^{s-1} H_i}$;
    otherwise $G = \Xor_{i=1}^{s} H_i$, where in both cases all $H_i$'s are from $\cX$, as required.
\end{proof}

Using \cref{obs:xor-of-k-graphs}, the above lemma can equivalently be stated as follows.

\begin{lemma}\label{lem:X-xor-Y}
    Let $\cX$ be an intersection-closed class of graphs, and $\cY$ be a $k$-function of $\cX$.
    Then for every graph $G \in \cY$ there exist $\alpha \in \zo$ and $s$ graphs $H_1, H_2, \ldots, H_s \in \cX$, where $s \leq 2^k$, such that
    for any two distinct vertices $a,b \in V(G)$ 
    \[
        (a,b) \in E(G) \quad \Longleftrightarrow \quad 
        | \{i ~:~ (a,b) \in E(H_i), i \in [s] \}| \equiv \alpha \pmod{2}.
    \]
    Furthermore, if $\cX$ contains all complete graphs, then $\alpha$ can always be chosen to be 1.
\end{lemma}

\subsection{Monotone graph classes}
\label{sec:monotone-classes}

In this section, we consider functions of monotone graph classes that come from families closed under the unions.

\begin{lemma}\label{lem:union-closed}
    Let $\cF$ be a family of \emph{monotone} graph classes such that for every $s \in \bN$ and $\cX_1,\ldots,\cX_s\in\cF$, there exists a class $\cX' \in \cF$ so that $\bigvee_{i=1}^s \cX_i \subseteq \cX'$.
    Then, for any $\cX \in \cF$ and any graph class $\cY$ that is a function of $\cX$, there exists $\cZ \in \cF$ such that $\cY \subseteq \cZ \cup \overline{\cZ}$.
\end{lemma}
\begin{proof}
    Let $\cX$ be a class from $\cF$ and let $\cY$ be a $k$-function of $\cX$ for some $k \in \bN$.
    Since $\cX$ is monotone and every monotone class is intersection-closed,
    by \cref{lem:intersection-closed-xor-representation}, for every $G \in \cY$ there exist $H_1, \ldots, H_s \in \cX$, for some $s \leq 2^k$, such that $G=\Xor_{i=1}^s H_i$ or $G=\overline{\Xor_{i=1}^s H_i}$.

    Note that $\Xor_{i=1}^s H_i$ is a subgraph of $\bigvee_{i=1}^s H_i = \bigcup_{i=1}^s H_i$.
    Let $\cZ$ be a class of graphs in $\cF$ such that $\bigvee_{i=1}^{2^k} \cX \subseteq \cZ$. Since $\cX$ is monotone, $\bigvee_{i=1}^{2^k} \cX$ is also monotone, and hence the graph $\Xor_{i=1}^s H_i$ belongs to $\bigvee_{i=1}^{2^k} \cX \subseteq \cZ$.
    Consequently, we have that every graph $G \in \cY$ belongs to $\cZ$ or $\overline{\cZ}$, i.e., $\cY \subseteq \cZ \cup \overline{\cZ}$ as claimed.    
\end{proof}

For $k \in \bN$ denote by $\cD_k$ the class of graphs of maximum degree at most $k$. Clearly, $\cD_k$ is a monotone class, and for any $s \in \bN$ the class $\bigvee_{i=1}^s \cD_k$ is a subclass of $\cD_{s \cdot k}$. Thus, by applying \cref{lem:union-closed} to the family of graph classes $\cF = \{ \cD_k : k \in \bN \}$, we derive the following. 

\begin{corollary}[Functions of bounded degree classes]\label{cor:bdd-deg}
    Let $\cX$ be a graph class of bounded degree and $\cY$ be a function of $\cX$. Then, there exists $k\in \bN$ such that $\cY \subseteq \cD_k \cup \overline{\cD_k}$.
\end{corollary}

For the next three corollaries the argument is similar, but requires a bit more detail.

\begin{corollary}[Functions of bounded degeneracy classes]\label{cor:bdd-dgn}
    Let $\cX$ be a graph class of bounded degeneracy and $\cY$ be a function of $\cX$. Then, there exists a class $\cZ$ of bounded degeneracy such that $\cY \subseteq \cZ \cup \overline{\cZ}$.
\end{corollary}
\begin{proof}
    For $k \in \bN$ denote by $\cR_k$ the class of graphs of degeneracy at most $k$. It is easy to see from the definition of degeneracy that every $\cR_k$ is a monotone class. 
    
    We claim that for any $s \in \bN$ the class $\bigvee_{i=1}^s \cR_k$ has degeneracy at most $2sk$.
    To see this, first observe that if a graph $G$ has degeneracy at most $k$, then any subgraph $H$ of $G$ has at most $k |V(H)|$ edges (this can be seen by starting with $H$ and iteratively removing a vertex of degree at most $k$ until no vertex is left). This property implies that for any $s$ graphs $H_1, H_2, \ldots, H_s \in \cR_k$
    every subgraph $H$ of the graph $\bigvee_{i=1}^s H_i = \bigcup_{i=1}^s H_i$ has at most $s k |V(H)|$ edges, which in turn implies that $H$ has a vertex of degree at most $2sk$. Thus the degeneracy of $\bigvee_{i=1}^s H_i$ is at most $2sk$, and therefore $\bigvee_{i=1}^s \cR_k \subseteq \cR_{2sk}$.
    
    By applying \cref{lem:union-closed} to the family of graph classes $\cF = \{ \cR_k : k \in \bN \}$ we derive the statement.
\end{proof}

\begin{corollary}[Functions of graphs of bounded chromatic number]\label{cor:bounded-chromatic-number}
    Let $\cX$ be a graph class of bounded chromatic number, and $\cY$ be a function of $\cX$. Then there exists a class $\cZ$ of bounded chromatic number such that $\cY \subseteq \cZ \cup \overline{\cZ}$.
\end{corollary}
\begin{proof}
    For $k \in \bN$ denote by $\cC_k$ the class of graphs of chromatic number at most $k$. Removing an edge or a vertex from a graph cannot increase its chromatic number, so every $\cC_k$ is a monotone class. 
    It is also easy to see that for any $s \in \bN$ the graphs in the class $\bigvee_{i=1}^s \cC_k$ have chromatic number at most $k^s$.
    By applying \cref{lem:union-closed} to the family of graph classes $\cF = \{ \cC_k : k \in \bN \}$ we derive the statement.
\end{proof}
Recall that a class $\mathcal{C}$ is weakly sparse if there exists a $t\in \mathbb{N}$ such that each graph $G\in \mathcal{C} $  excludes the biclique $K_{t,t}$ as a subgraph. 

\begin{corollary}[Functions of weakly-sparse classes]\label{cor:weakly-sparse}
    Let $\cX$ be a weakly-sparse graph class and $\cY$ be a function of $\cX$. Then there exists a weakly-sparse class $\cZ$ such that $\cY \subseteq \cZ \cup \overline{\cZ}$.
\end{corollary}
\begin{proof}
    For $k \in \bN$ denote by $\cQ_k$ the class of graphs with biclique number at most $k$. As edge or vertex removal cannot increase the biclique number of a graph, every $\cQ_k$ is a monotone class.
    
    We claim that for any $s \in \bN$ the class $\bigvee_{i=1}^s \cQ_k$ is weakly-sparse, i.e., the biclique number is bounded in the class.
    To see this, first denote by $R(t,c)$ the bipartite Ramsey number, i.e., the minimum number $p$ such that any edge coloring of $K_{p,p}$ in $c$ colors contains a monochromatic $K_{t,t}$. Suppose now that $\bigvee_{i=1}^s \cQ_k$ is not weakly-sparse, and let $G = \bigvee_{i=1}^s H_i$, $H_i \in \cQ_k, i \in [s]$, be a graph from the class with the biclique number at least $R=R(k+1,s)$, witnessed by a biclique $G[A,B]$ of size $R$. Color each edge of $G[A,B]$ by one of $s$ colors depending on which of the $s$ graphs $H_1, H_2, \ldots, H_s$ the edge belongs to (if the edge belongs to more than one graph, choose one of the respective colors arbitrarily). By Ramsey's theorem, there exist sets $A' \subseteq A$ and $B' \subseteq B$ such that $G[A',B']$ is a monochromatic biclique of size $k+1$. This means that for some $i \in [s]$, $H_i[A',B']$ is also a biclique, contradicting the assumption that $H_i$ belongs to $\cQ_k$.
    This contradiction implies that the class $\bigvee_{i=1}^s \cQ_k$ is weakly-sparse and in particular it is a subclass of $\cQ_R$.

    By applying \cref{lem:union-closed} to the family of graph classes $\cF = \{ \cQ_k : k \in \bN \}$ we derive the statement.
\end{proof}

\subsection{Hierarchy of classes of equivalence graphs}
\label{sec:equivalence-graphs}

In this section, we examine the expressive power
of Boolean functions of equivalence graphs.
We do this by revealing a hierarchy of hereditary classes of equivalence graphs with respect to the expressiveness of their Boolean functions (see \cref{fig:hierarchy}).
This hierarchy is complete in the sense that a class in the hierarchy cannot be expressed as a Boolean function of any class below it, but every proper subclass can be. This holds true for every class in the hierarchy except the class of graphs with at most one edge ($\cE_1$), for which the property almost holds: any proper subclass $\cX$ of $\cE_1$ becomes a function of edgeless graphs ($\mathcal{E}_0$) after removing finitely many graphs from $\cX$.

In addition to $\cD_1$, which as before denotes the class of graphs of degree at most 1, in \cref{fig:hierarchy},
the calligraphic letters denote the following classes:
\begin{enumerate}
    \item[$\cE_k$] is the class of graphs with at most $k$ edges.
    \item[$\cL$] is the class of equivalence graphs consisting of all complete graphs and graphs with two connected components, one of which is an isolated vertex.
    \item[$\cC$] is the class of equivalence graphs with at most one connected component of size more than 1, i.e., each graph in $\cC$ consists of a clique and isolated vertices.
\end{enumerate}

As we shall see, some known families of graph classes can be characterized as Boolean combinations of restricted classes of equivalence graphs.
For example, functions of $\cE_1$ are precisely classes consisting of graphs with a bounded number of edges or a bounded number of non-edges; and functions of $\cC \cup \cD_1$ characterize graph classes of structurally bounded degree.

Boolean functions of some classes in the hierarchy characterize families of hereditary graph classes with certain speed.
It is known \cite{SZ94,Ale97,BBW00} that the speed of a hereditary graph class $\cX$ cannot be arbitrary. In particular, if $|\cX^n|$ grows as $2^{O(n \log n)}$, then it grows as either $\Theta(1)$ or $n^{\Theta(1)}$ or $2^{\Theta(n)}$ or $2^{\Theta(n \log n)}$.
In these cases class $\cX$ is said to be constant, polynomial, exponential, and factorial, respectively.
The Boolean functions of the class $\cL$ are precisely subexponential classes, i.e., classes with speed $2^{o(n)}$; and the Boolean functions of the class $\cC$ characterize subfactorial classes, i.e., classes with speed $2^{o(n \log n)}$.

\begin{figure}
    \centering 
    \begin{tikzpicture}[node distance=3.5cm, thick]
        \node[draw, ultra thin, fill=yellow!10!white,  text=purple, minimum width=3cm, minimum height=2cm, rounded corners, align=left] (A) at (0,0) {\parbox{3cm}{$\cE_0$}};
        \pic[scale=0.75] at (-0.5, 0.7) {emptygraph};
        \node[right=0.5cm of A, align=center] {\parbox{3.5cm}{Empty or complete}};
        
        \node[draw, ultra thin, fill=yellow!10!white, text=purple, minimum width=3cm, minimum height=2cm, rounded corners, align=left] (B) at (0,3) {\parbox{3cm}{$\cE_1$}};
        \pic[scale=0.75] at (-0.5, 3.7) {oneedgegraph};
        \node[right=0.5cm of B, align=center] {\parbox{3.5cm}{Bounded $\#$ of \\ edges or non-edges}};
        
        \node[draw, ultra thin, fill=yellow!10!white, text=purple, minimum width=3cm, minimum height=2cm, rounded corners, align=left] (C)  at (-3,6) {\parbox{3cm}{$\cL$}};
        \pic[scale=0.75] at (-3.3, 6) {cliqueplusvertex};
        \node[left=0.5cm of C, align=center] {\parbox{2.3cm}{Twin class of \\ size $n-O(1)$; \\ Subexponential \\ speed}};
 
        \node[draw, ultra thin, fill=yellow!10!white, text=purple, minimum width=3cm, minimum height=2cm, rounded corners, align=left] (D) at (3,6) {\parbox{3cm}{$\cD_1$}};
        \pic[scale=0.75] at (2.7, 6) {degone};
        \node[right=0.5cm of D, align=center] {\parbox{3.5cm}{Bounded \\ degree or co-degree}};
 
        \node[draw, ultra thin, fill=yellow!10!white, text=purple, minimum width=3cm, minimum height=2cm, rounded corners, align=left] (E) at (-3,9) {\parbox{3cm}{$\cC$}};
        \pic[scale=0.65] at (-3.3, 9) {cliqueplusis};
        \node[left=0.5cm of E, align=center] {\parbox{2.3cm}{Bounded $\#$ of\\ twin classes; \\ Subfactorial \\ speed}};
 
        \node[draw, ultra thin, fill=yellow!10!white, text=purple, minimum width=3cm, minimum height=2cm, rounded corners, align=center] (F) at (3,9) {$\cL \cup \cD_1$};

        \node[draw, ultra thin, fill=yellow!10!white, text=purple, minimum width=3cm, minimum height=2cm, rounded corners, align=center] (G) at (0,12) {$\cC \cup \cD_1$};
        \node[right=0.5cm of G, align=center] {\parbox{3cm}{Structurally \\bounded degree}};
 
        \node[draw, ultra thin, fill=yellow!10!white, text=black, minimum width=3cm, minimum height=2cm, rounded corners, align=center] (H) at (0,15) {};
        \node[align=center] at (0.15,15.7) {\parbox{3cm}{\scalebox{0.85}{Equivalence graphs}}};
        \pic[scale=0.4] at (-0.23, 14.85) {equivgraph};
        \node[right=0.5cm of H, align=center] {\parbox{4cm}{Bounded partition \\ complementation \\ number}};
 
        \draw[->, dashed, line width=.2mm] (A) -- (B);
        \draw[->, line width=.2mm] (B) -- (C); 
        \draw[->, line width=.2mm] (B) -- (D); 
        \draw[->, line width=.2mm] (C) -- (F);
        \draw[->, line width=.2mm] (C) -- (E);
        \draw[->, line width=.2mm] (D) -- (F);
        \draw[->, line width=.2mm] (E) -- (G); 
        \draw[->, line width=.2mm] (F) -- (G);
        \draw[->, line width=.2mm] (G) -- (H);
    \end{tikzpicture}
    \caption{A hierarchy of classes of equivalence graphs with respect to the expressiveness of their Boolean functions. Each node in the hierarchy corresponds to a hereditary class of equivalence graphs. For each class in the hierarchy, except $\cE_1$, every proper subclass of this class is a Boolean function of one of the classes immediately below it; the same almost holds for $\cE_1$, namely, any proper subclass of $\cE_1$ becomes a function of $\cE_0$ after removing finitely many graphs from it (this exceptional relation is reflected in the picture with the dashed arrow).
    The text beside a node explains characterization(s) of Boolean functions of the corresponding class.}
    \label{fig:hierarchy}
\end{figure}
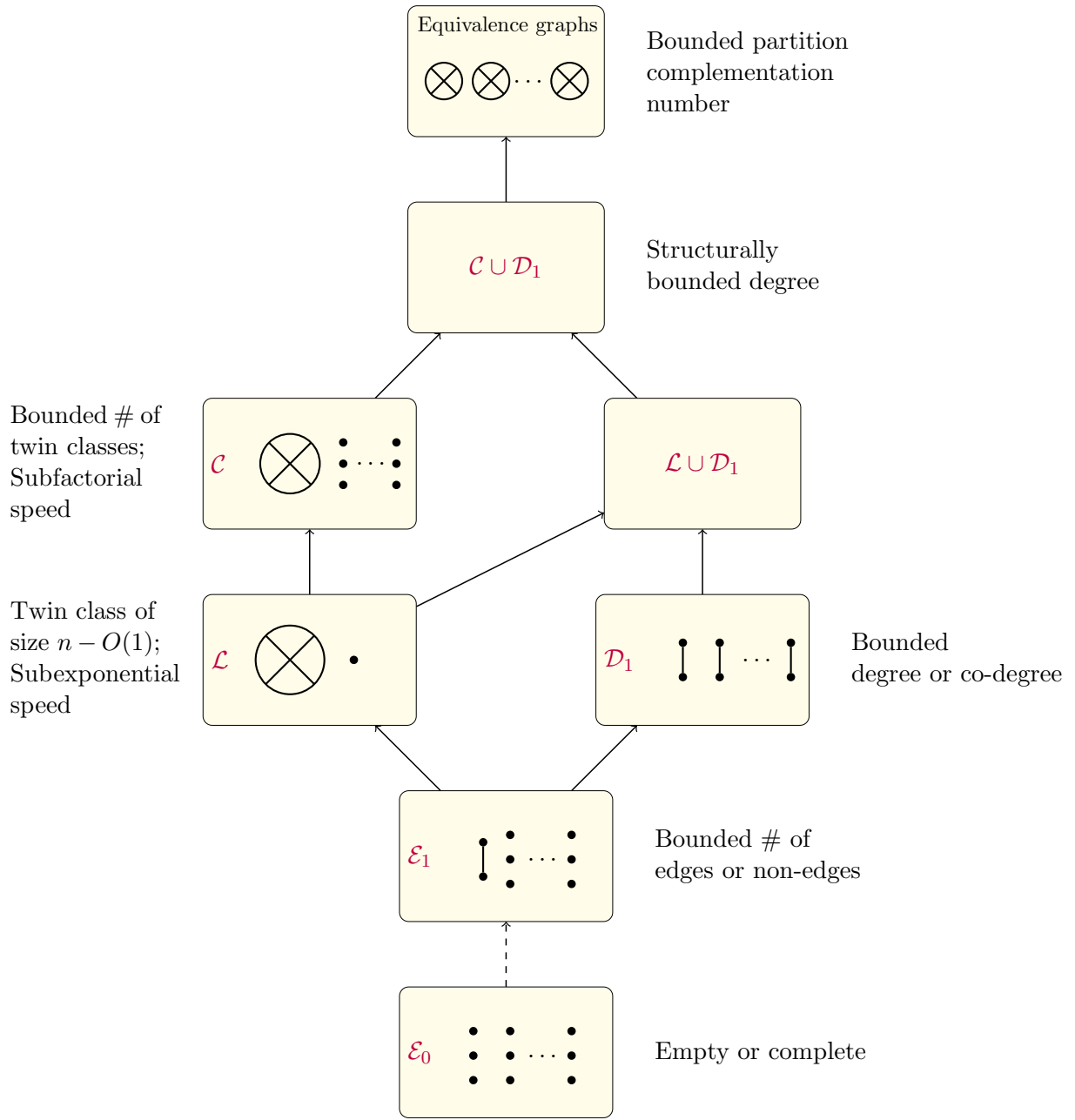

\subsubsection{Empty graphs}

\begin{theorem}\label{th:empty}
    A class of graphs is a function of $\cE_0$
    if and only if each of its graphs is either complete or empty.
\end{theorem}
\begin{proof}
    Clearly, each empty or complete $n$-vertex graph is a function of an $n$-vertex empty graph; in the former case, the corresponding Boolean function is $f(x) = x$, and in the latter case it is $f(x) = \overline{x}$.
    
    Conversely, let $\cX$ be a $k$-function of the class $\cE_0$ of empty graphs for some $k \in \bN$. Since $\cE_0$ is intersection-closed, by \cref{lem:intersection-closed-xor-representation}, each graph $G=(V,E)$ in $\cX$ is equal to
    $\Xor_{i=1}^s H_i$ or $\overline{\Xor_{i=1}^s H_i}$, where $s \leq 2^k$ and $H_1, H_2, \ldots, H_s$ are some 
    graphs from $\cE_0$ on the same vertex set $V$. Since $\cE_0$ has only one graph on $V$, all graphs $H_i, i \in [s]$, are equal to this unique empty graph. Therefore, 
    $\Xor_{i=1}^s H_i$ is the empty graph on $V$ and $\overline{\Xor_{i=1}^s H_i}$ is the complete graph on $V$. 
\end{proof}

\subsubsection{Graphs with at most one edge}

In this section, we show that graph classes in which graphs have bounded number of edges or bounded number of non-edges are precisely functions of the graphs with at most $1$ edge.

\begin{theorem}\label{th:one-edge}
    A class of graphs $\cX$ is a function of $\cE_1$ if and only if there exists a constant $t$ such that every graph in $\cX$ has at most $t$ edges or at most $t$ non-edges.
\end{theorem}
\begin{proof}
    Let $\cX$ be a class of graphs such that there exists a constant $t \in \bN$ and every graph in $\cX$ has at most $t$ edges or at most $t$ non-edges. We claim that $\cX$ is a $t$-function of $\cE_1$. Let $G=(V,E)$ be a graph in $\cX$. Assume first that $G$ has at most $t$ edges. Then $G$ can be expressed as the disjunction $\bigvee_{e \in E} H_e$ of $|E| \leq t$ graphs, where $H_e := (V, \{e\})$. Suppose now that $G$ has at most $t$ non-edges. Then $\overline{G} = (V, \overline{E})$, where $\overline{E} = \binom{V}{2} \setminus E$, has  at most $t$ edges, and thus, as before, $\overline{G} = \bigvee_{e \in \overline{E}} H_e$ and thus $G = \overline{\bigvee_{e \in \overline{E}} H_e}$ is a function of at most $|\overline{E}| \leq t$ graphs from $\cE_1$.

    To prove the converse, we observe that $\cE_k$ is a monotone class and for any $s \in \bN$ the class $\bigvee_{i=1}^s \cE_k$ is equal to $\cE_{s \cdot k}$. Thus, applying \cref{lem:union-closed} to the family of graph classes $\cF = \{ \cE_k : k \in \bN \}$, we conclude that for any graph class $\cX$ that is a function of $\cE_1$ there exists a $t \in \bN$ such that $\cX \subseteq \cE_t \cup \overline{\cE_t}$.
\end{proof}

\subsubsection{Clique + isolated vertex}

In this section we characterize graph classes that are functions of $\cL$.

Let $G$ be a graph. We say that two vertices $a,b$ in $G$ are \emph{twins} if $N(a) \setminus \{a,b\} = N(b) \setminus \{a,b\}$. It is easy to see that the twin relation is an equivalence relation, and each equivalence class induces a clique or an independent set in $G$. We refer to an equivalence class of the twin relation as a \emph{twin class}.
We call the number of twin classes the \emph{twin number} of $G$.

\begin{theorem}\label{th:clique+isolated}
    A class of graphs $\cX$ is a function of $\cL$ if and only if there exists an integer $k \geq 0$ such that every graph in $\cX$ has a twin class that contains all but at most $k$ vertices.
\end{theorem}
\begin{proof}
    Suppose there exists an integer $k \geq 0$ such that every graph in $\cX$ has a twin class containing all but at most $k$ vertices. Let $G=(V,E)$ be a graph in $\cX$, and let $Q \subseteq V$ be a twin class in $G$ containing all but $p \leq k$ vertices. Then $Q$ is either a clique or an independent set, and every vertex in $P = V \setminus Q$ is either anticomplete or complete to $Q$. We claim that $G$ is a function of at most $p$ graphs from $\cL$.

    To show this, we first introduce some notation. 
    Denote by $P_1$ the set of vertices in $P$ that are anticomplete to $Q$, and by $P_2$ the set of vertices in $P$ that are complete to $Q$. For a set of vertices $U \subseteq V$, denote by $C_U$ the equivalence graph in which $V \setminus U$ is a clique, and all vertices in $U$ are isolated; if $U=\{a\}$, we will write $C_a$ to denote $C_U$. Note that for every $a \in V$, graph $C_a$ is in $\cL$, and $C_U$ is equal to $\bigwedge_{a \in U} C_a$.
    Further, for a pair of distinct vertices $a,b \in V$, denote by $H_{a,b}$ the graph on vertex set $V$ with a unique edge $(a,b)$, and note that $H_{a,b} = \overline{C_a} \wedge \overline{C_b}$.

    Let $G' = C_{P_1} \wedge \overline{C_{P_2}}$.
    In this graph, every vertex in $P_1$ is anticomplete to $Q$, every vertex in $P_2$ is complete to $Q$, and $Q$ is an independent set.
    If $Q$ is also an independent set in $G$, then we define $G'' = G'$. Otherwise, if 
    $Q$ is a clique in $G$, then we define $G'' = G' \vee C_P$.

    Graph $G''$ is almost the same as $G$. Namely, any two vertices in $V$ that are not both in $P$ are adjacent in $G''$ if and only if they are adjacent in $G$. In order to turn $G''$ into $G$ we observe that for any graph $F=(V,E)$ and any distinct $a,b \in V$, if $a,b$ are not adjacent in $F$, then taking the union of $F$ with $H_{a,b}$ adds the edge $(a,b)$ to $F$ (and nothing else); and if $a,b$ are adjacent in $F$, then intersecting $F$ with $\overline{H_{a,b}}$ removes the edge $(a,b)$ from $F$.

    Thus, denoting by $E_1 \subseteq \binom{P}{2}$ the set of pairs of vertices in $P$ that are not adjacent in $G''$, but adjacent in $G$; and 
    by $E_2 \subseteq \binom{P}{2}$ the set of pairs that are adjacent in $G''$, but not adjacent in $G$, we conclude that
    \[
        G = \left( G'' \vee \bigvee_{\{a,b\} \in E_1} H_{a,b} \right) \wedge \bigwedge_{\{a,b\} \in E_2} \overline{H_{a,b}}.
    \]
    Since each of the graphs in the above formula is a Boolean combination of graphs in $\{C_a~|~a \in P\} \subseteq \cL$, we conclude that $G$ is a function of $p$ graphs from $\cL$. 

    Now, suppose that $\cX$ is a $k$-function of $\cL$ for some $k \in \bN$. We will show that every graph in $\cX$ has a twin class containing all but at most $k$ vertices.
    Let $G=(V,E)$ be an $n$-vertex graph in $\cX$, and suppose that $G$ is a function of $k$ graphs $C_1, C_2, \ldots, C_k \in \cL$. Each of the graphs $C_1, C_2, \ldots, C_k$ has a twin class of size $n-1$. For $i \in [k]$, denote by $W_i$ the corresponding twin class in $C_i$. Then $W = \bigcap_{i \in [k]} W_i$ is a set of pairwise twins in each of the graphs $C_1, C_2, \ldots, C_k$.
    We claim that $W$ is also a set of pairwise twins in $G$.
    Indeed, for any two distinct $a,b \in W$, and $c \in V \setminus \{a,b\}$ we have that $(C_1(a,c), C_2(a,c), \ldots, C_k(a,c)) = (C_1(b,c), C_2(b,c), \ldots, C_k(b,c))$, and thus $G(a,c) = G(b,c)$. Hence, $N_G(a) \setminus \{a,b\} = N_G(b) \setminus \{a,b\}$, i.e., $a$ and $b$ are twins in $G$.
    Consequently, $G$ contains a twin class of size at least $|W| \geq n-k$, as desired.
\end{proof}

It is known that a hereditary class $\cX$ is subexponential if and only if there exists $k \in \bN$ such that every graph in $\cX$ has a twin class containing all but at most $k$ vertices \cite{SZ94,Ale97,BBW00}. Combining this fact with \cref{th:clique+isolated} we obtain a characterization of subexponential graph classes as functions of the class $\cL$.

\begin{theorem}\label{th:clique-isolated-vertex-subexponential}
    A hereditary graph class is subexponential if and only if it is a function of $\cL$.
\end{theorem}

\subsubsection{Graphs of degree at most 1}

In this section, we show that graph classes in which graphs have bounded degree or bounded co-degree are precisely functions of the graphs of maximum degree at most 1.
Recall, for $k \in \bN$, $\cD_k$ denotes the class of graphs of maximum degree at most $k$.

\begin{theorem}\label{th:degree-1}
    A class of graphs $\cX$ is a function of $\cD_1$ if and only if there exists a $k \in \bN$ such that $\cX \subseteq \cD_k \cup \overline{\cD_k}$.
\end{theorem}
\begin{proof}
    By \cref{cor:bdd-deg}, if $\cX$ is a function of $\cD_1$, then $\cX \subseteq \cD_k \cup \overline{\cD_k}$ for some $k \in \bN$.

    To prove the converse, suppose there exists $k \in \bN$ such that for every graph $G \in \cX$ we have $\Delta(G) \leq k$ or $\Delta(\overline{G}) \leq k$.
    We let $F=(V,E)$ denote the graph $G$ if $\Delta(G) \leq k$, and the graph $\overline{G}$ otherwise.
    By Vizing's theorem \cite{Viz64}, the edges of $F$ can be colored in at most $\Delta(F)+1 \leq k+1$ colors in such a way that no two edges of the same color share a vertex. In other words, the edge set $E$ of $F$ can be partitioned into at most $k+1$ matchings. Let $M_1, M_2, \ldots, M_s \subseteq E$, where $s \leq k+1$, be the matchings of such a partition of $E$. Then $F = \bigvee_{i=1}^s H_i$, where $H_i = (V,M_i)$ for $i \in [s]$. Consequently, $G = \bigvee_{i=1}^s H_i$ if $\Delta(G) \leq k$, or, otherwise, $G = \overline{\bigvee_{i=1}^s H_i}$. Since the maximum degree of each of the graphs $H_1, H_2, \ldots, H_s$ is at most one, we conclude that $\cX$ is a $(k+1)$-function of~$\cD_1$.
\end{proof}

\subsubsection{Clique + isolated vertex and graphs of degree at most 1}

Let $\XY_k$ be the class of graphs where all but at most $k$ vertices have degree at most $k$, and the remaining vertices have co-degree at most $k$. In this section we show that the classes that are functions of  $\mathcal{L} \cup \mathcal{D}_1$ are precisely those contained in $\XY_k \cup \overline{\XY_k}$ for some $k$.

\begin{lemma} \label{lem:XY}
    Let $k,r \in \mathbb{N}$, and let $f$ be an $r$-variable Boolean function. Then $f(\XY_k) \subseteq \XY_{s} \cup \overline{\XY_{s}}$, where $s=kr2^{r}$.
\end{lemma}
\begin{proof}
    We start by showing that, for any $k_1,k_2 \in \bN$, we have
    $\XY_{k_1} \wedge \XY_{k_2} \subseteq \XY_{k_1+k_2}$ and $\XY_{k_1} \xor \XY_{k_2} \subseteq \XY_{k_1+k_2}$.
    Let $G_1 \in \XY_{k_1}$ and $G_2 \in \XY_{k_2}$ be graphs on the same vertex set $V$, and for $i \in \{1,2\}$, let $B_i$ be the set of vertices of 
    co-degree at most $k_i$ in $G_i$, and denote $A_i = V \setminus B_i$.

    First, consider $H = G_1 \wedge G_2$ and a vertex $v \in V$.
    If $v \in A_1$ or  $v \in A_2$, then it has degree at most $\max\{k_1, k_2\}$ in $H$. If $v \in B_1 \cap B_2$, then it has co-degree at most $k_1+k_2$ in $H$. Since $|B_1 \cap B_2| \leq \min\{k_1, k_2\}$, we conclude that $H \in \XY_{k_1+k_2}$.

    Next, consider $H = G_1 \xor G_2$ and a vertex $v \in V$. If $v \in A_1 \cap A_2$ or $v \in B_1 \cap B_2$, then $v$ has degree at most $k_1+k_2$ in $H$. If $v \in A_1 \cap B_2$ or $v \in A_2 \cap B_1$, then $v$ has co-degree at most $k_1+k_2$ in $H$.
    Since $|A_1 \cap B_2| + |A_2 \cap B_1| \leq k_1+k_2$, we conclude that $H \in \XY_{k_1+k_2}$.

    Now, let $G$ be an arbitrary graph in $f(\XY_k)$.
    It follows from \cref{th:anf} that
    there exist $r$ graphs $H_1, H_2, \ldots, H_r \in \XY_k$ and
    distinct non-empty sets $I_1, I_2, \ldots, I_q \subseteq [r]$ such that either $G$ or $\overline{G}$ is equal to
    $$
        \Xor_{j=1}^{q} \bigwedge_{i \in I_j} H_i.
    $$
    Consequently, by the above, either $G$ or $\overline{G}$ belongs to $\XY_{s}$, where $s = \sum_{i=1}^{q} k \cdot |I_i| \leq kr 2^r$, as desired.
\end{proof}

\begin{theorem}
    A class of graphs $\cX$ is a function of $\mathcal{L} \cup \mathcal{D}_1$ if and only if there exists a $k \in \bN$ such that $\cX \subseteq \XY_k \cup \overline{\XY_k}$.    
\end{theorem}
\begin{proof}
Observe that if $G\in \mathcal{L}$ then $\overline{G}\in \cD_1^*$. Thus,  since every graph from $\mathcal{X}$ is a function of graphs from $\cD_1$ and $\mathcal{L}$, viewing the latter graphs as complements of graphs in $ \cD_1^*$, we have that $\mathcal{X}$ is a function of $   \cD_1\cup \cD_1^* \subseteq  \cD_1^*$.  It follows from \cref{lem:XY} that if $\cX$ is a function of $\cD_1^*$, then $\cX \subseteq \XY_k \cup \overline{\XY_k}$ for some $k \in \bN$.

    To prove the converse, suppose there exists $k \in \bN$ such that $\cX \subseteq \XY_k \cup \overline{\XY_k}$. Let $G$ be a graph from $\cX$ that belongs to $\XY_k$. Then $G$ can be obtained from a graph of degree at most $k$ (which is a union of at most $k+1$ graphs from $\mathcal{D}_1$) by taking a union with at most $k$ stars (each of which is the complement of a graph in $\mathcal{L}$) and then taking an XOR with at most $k^2$ single-edge graphs (each of which is the intersection of two graphs that are complements of graphs in $\mathcal{L}$). Thus, $G$ is a function of at most $2k^2+2k+1$ graphs from $\mathcal{L} \cup \mathcal{D}_1$.
    
    If $G$ belongs to $\overline{\XY_k}$, then by taking an XOR with the complete graph (which belongs to $\cL$) we obtain $\overline{G} \in \XY_k$.
    Therefore any $G\in \cX$ is a function of at most $2k^2+2k+2$ graphs from $\mathcal{L} \cup \mathcal{D}_1$.
\end{proof}

\subsubsection{Clique + independent set}

In this section we characterize graph classes that are functions of $\cC$.

We note that \cite{BuchananPR22,PouzetKT21} study XORs of graphs in $\cC$. Since $\cC$ is an intersection-closed class and contains all complete graphs, due to \cref{lem:intersection-closed-xor-representation}, XORs of $\cC$ are as expressive as Boolean functions of $\cC$.

\begin{lemma}\label{lem:subgraph-complementation-twin-classes}
    Let $G=(V,E)$ be a graph and let $H=(V,E')$ be a subgraph complementation of $G$. If the twin number of $G$ is $k$, then the twin number of $H$ is at most $2k$.
\end{lemma}
\begin{proof}
    Let $U \subseteq V$ be such that $H$ is the subgraph complementation of $G$ with respect to $U$. 

    If $a,b \in U$ are twins in $G$, they stay twins in $H$ as the complementation of the edges in the subgraph $G[U]$ changes the neighborhood of the two vertices in the same way. Specifically, 
    \[
        N_H(a) \setminus \{b\} 
        = (N_G(a) \setminus U) \cup (U \setminus (N_G(a) \cup \{a,b\}))
        = (N_G(b) \setminus U) \cup (U \setminus (N_G(b) \cup \{a,b\})) 
        = N_H(b) \setminus \{a\}.
    \]
    If two vertices $a,b \not\in U$ are twins in $G$,
    their neighborhoods are not affected by the complementation of the subgraph $G[U]$, and thus 
    they remain twins in $H$.

    It follows from the above that for any twin equivalence class $R \subseteq V$ of $G$, each of the sets $R \cap U$ and $R \setminus U$ consists of pairwise twins. Consequently, $H$ has at most twice as many twin equivalence classes as $G$ does, which completes the proof.
\end{proof}

The next theorem characterizes graph classes of bounded twin number as functions of $\cC$.

\begin{theorem}\label{th:clique-is}
    Let $\cX$ be a class of graphs. Then, the twin number is bounded in $\cX$ if and only if $\cX$ is a function of $\cC$.
\end{theorem}
\begin{proof}
    Suppose $\cX$ is a $k$-function of $\cC$ for some $k \in \bN$. We claim that the twin number of graphs in $\cX$ is at most $2^{2^k}$.
    
    Let $G=(V,E)$ be a graph in $\cX$. 
    First note that $\cC$ is intersection-closed and contains all complete graphs; hence, by \cref{lem:intersection-closed-xor-representation}, there exist $s \leq 2^k$ and $H_1, H_2, \ldots, H_s \in \cC$ such that
    $G = \Xor_{i=1}^s H_i$.
    Since $\Xor_{i=1}^s H_i = (\Xor_{i=1}^{s-1} H_i) \xor H_s$, \cref{obs:xor-with-clique} implies that $G$ is obtained from the empty graph $(V, \emptyset)$ via the sequence of subgraph complementations associated with $H_1, H_2, \ldots, H_s$. Thus, from \cref{lem:subgraph-complementation-twin-classes}, we conclude that the number of equivalence classes in $G$ is at most $2^{s} \leq 2^{2^k}$.

    Conversely, let $\cX$ be a class of graphs, and suppose there exists $s \in \bN$ such that the twin number of every graph in $\cX$ is at most $s$. We claim that $\cX$ is a $k$-function of $\cC$, where $k = \binom{s}{2} + s$.

    Let $G=(V,E)$ be a graph in $\cX$ and let $t$ be the twin number of $G$, where $t \leq s$.
    Denote by $V_1, V_2, \ldots, V_t$ the twin equivalence classes of $G$.
    Note that for every $i \in [t]$, the class $V_i$ is a clique or an independent set; and for every distinct $i,j \in [t]$ either $V_i$ is complete to $V_j$ or $V_i$ is anticomplete to $V_j$.
    Let $J \subseteq \binom{[t]}{2}$ with $\{i,j\} \in J$ if and only if $V_i$ is complete to $V_j$.
    For every $\{i,j\} \in J$, let $H_{\{i,j\}}$ be the graph in $\cC$ on vertex set $V$, in which $V_i \cup V_j$ is a clique and all other vertices are isolated.
    Let $H = \bigvee_{e \in J} H_{e}$. By construction, for any two distinct $i,j \in [t]$, any two vertices $a \in V_i$ and $b \in V_j$ are adjacent in $G$ if and only if they are adjacent in $H$. Furthermore, for each $i \in [t]$, the set $V_i$ is a clique or an independent set in $H$. In order to make $H$ equal to $G$ we need to make subgraph complementation of $H$ with respect to each $V_i$, where $H$ and $G$ disagree. More formally, let $I \subseteq [t]$ be the set of indices $i$ such that $V_i$ is a clique in $G$ and an independent set in $H$, or $V_i$ is an independent set in $G$ and a clique in $H$. 
    For $i \in [t]$, let $H_i$ be the graph in $\cC$ on vertex set $V$, in which $V_i$ is a clique and all other vertices are isolated. 
    By \cref{obs:xor-with-clique}, the subgraph complementation of $H$ with respect to $V_i$ is equal to $H \xor H_i$. Thus, we have that $G = 
    H \xor \Xor_{i \in I} H_i$.
    Therefore, $G$ is a function of $|J|+|I|$ graphs from $\cC$, which implies the required claim as $|J| + |I| \leq \binom{t}{2} + t \leq \binom{s}{2} + s$.
\end{proof}

A hereditary graph class $\cX$ is called \emph{subfactorial} if the number $|\cX^n|$ of labeled $n$-vertex graphs in $\cX$ grows as $2^{o(n\log n)}$. It is known that a hereditary class $\cX$ is subfactorial if and only if the twin number is bounded in $\cX$ \cite{SZ94,Ale97,BBW00}. Combining this fact with \cref{th:clique-is} we obtain a characterization of subfactorial graph classes as functions of the class $\cC$.

\begin{theorem}\label{th:clique-is-subfactorial}
    A hereditary graph class is subfactorial if and only if it is a function of $\cC$.
\end{theorem}

\subsubsection{Clique + independent set and graphs of degree at most 1}
\label{sec:clique+deg1}

In this section we characterize functions of $\cC \cup \cD_1$ as graph classes of structurally bounded degree, which are defined in terms of first-order transductions.

We omit the formal definition of first-order transductions and refer the interested reader to e.g.~\cite{braunfeld2025first}.
Informally, given a first-order formula $\phi(x,y)$ and a graph $G$, a \emph{first-order $\phi$-transduction} of $G$ is
a transformation of $G$ to 
a graph that is obtained from $G$ by first taking a constant number of vertex-disjoint copies of $G$, 
then coloring the vertices of the new graph by a constant number of colors, 
then using $\phi$ as the new adjacency relation, 
and finally taking an induced subgraph.
A graph class $\cY$ is a first-order transduction of a graph class $\cX$ if there exists a first-order formula $\phi$ so that every graph in $\cY$ is a $\phi$-transduction of a graph from $\cX$.

\begin{definition}
    A graph class $\cX$ has \emph{structurally bounded degree} if it is obtainable by a first-order transduction from a class of bounded degree.
\end{definition}

\begin{definition}
    Given a graph $G$, $\tau(G)$ is the least number such that $G$ can be obtained from a graph of maximum degree $\tau(G)$ by making at most $\tau(G)$-many successive subgraph complementations.
\end{definition}

\begin{lemma}\label{lem:tau-structurally-bounded-degree}
    A class $\cX$ is of structurally bounded degree if and only if $\tau$ is bounded in $\cX$.
\end{lemma}
\begin{proof}
    In \cite[Proposition 14.1]{braunfeld2025first}, it is stated that a class $\cX$ is of structurally bounded degree if and only if it is a perturbation of a class $\cY$ of bounded degree, where a perturbation is a composition of boundedly many subgraph complementations. So for every $G \in \cX$, we have $\tau(G)$ is at most the larger of the maximum degree of $\cY$ and the number of subgraph complementations in the perturbation yielding $\cX$ from $\cY$.
\end{proof}

The next theorem characterizes functions of $\cC \cup \cD_1$ as graph classes of structurally bounded degree. Thus, by the characterization in \cite{GHOLR20} of graph classes of structurally bounded degree as being \emph{near-uniform}, we see that in passing from functions of $\cC$ in the previous subsection to functions of $\cC \cup \cD_1$, we pass from graphs with bounded twin number to graphs with bounded near-$k$-twin number for some $k$ (i.e., the relation of two vertices having neighborhoods differing by at most $k$ vertices 
forms an equivalence relation with boundedly many classes).

\begin{theorem}\label{th:clique-degree-1}
    Let $\cX$ be a class of graphs. Then $\cX$ is of structurally bounded degree if and only if $\cX$ is a function of $\cC \cup \cD_1$.
\end{theorem}
\begin{proof}
    We will show that the parameter $\tau$ is bounded in $\cX$ if and only if $\cX$ is a function of $\cC \cup \cD_1$, which is equivalent to the statement due to \cref{lem:tau-structurally-bounded-degree}.
    
    Suppose $\cX$ is a $k$-function of $\cC \cup \cD_1$ for some $k \in \bN$. We claim that $\tau(G)$ is at most $2^k$ for every $G \in \cX$.
    Let $G=(V,E)$ be a graph in $\cX$. 
    Since $\cC \cup \cD_1$ is intersection-closed and contains all complete graphs, by \cref{lem:intersection-closed-xor-representation}, there exist $s \leq 2^k$ and $H_1, H_2, \ldots, H_s \in \cC \cup \cD_1$ such that $G = \Xor_{i=1}^s H_i$.
    Without loss of generality, assume that the first $s_1$, $0 \leq s_1 \leq s$, of these graphs are from $\cD_1$, denote them $D_i = H_i$, $i \in [s_1]$; and the remaining $s_2 = s-s_1$ graphs are from $\cC$, denote them $F_i = H_{s_1+i}$, $i \in [s_2]$.
    Then, we have 
    \[
        G = \Xor_{i=1}^{s_1} D_i \xor \Xor_{i=1}^{s_2} F_i.
    \]
    Denote by $D$ the graph $\Xor_{i=1}^{s_1} D_i$ and note that the maximum degree of $D$ is at most $s_1$, since the maximum degree of each $D_i$ is at most one.
    Thus, $G = D \xor \Xor_{i=1}^{s_2} F_i$, where $D \in \cD_{s_1}$, and therefore,
    by \cref{obs:xor-with-clique}, $G$ is obtained from $D$ via the sequence of subgraph complementations associated with $F_1, F_2, \ldots, F_{s_2} \in \cC$. Consequently, $G$ is obtained from a graph of maximum degree at most $s_1 \leq s$ by making at most $s_2 \leq s$ subgraph complementations, i.e., $\tau(G) \leq s \leq 2^k$.
    
    Conversely, suppose that parameter $\tau$ is bounded in $\cX$, i.e., there exists $k \in \bN$ such that every graph in $\cX$ can be obtained from a graph of degree at most $k$ by making at most $k$ subgraph complementations. 
    We claim that $\cX$ is a $(2k+1)$-function of $\cC \cup \cD_1$.
    Let $G \in \cX$ and let $D$ be a graph of maximum degree at most $k$ such that $G$ is obtained from $D$ by making at most $k$ subgraph complementations. Let $F_1, F_2, \ldots, F_r$, $r \leq k$ be the graphs from $\cC$ corresponding to these subgraph complementations.
    It follows from the proof of \cref{th:degree-1} that $D$ is a union of $q \leq k+1$ graphs $M_1, M_2, \ldots, M_q$ of degree at most $1$.
    Thus, $G = \left(\bigvee_{i=1}^q M_i\right) \xor \Xor_{i=1}^r F_i$, and the claim follows, since each of $F_1, \ldots, F_r$ belongs to $\cC$, and $M_1, \ldots, M_q$ belong to $\cD_1$.
\end{proof}

Finally, we characterize graph classes of structurally bounded degree as functions of proper subclasses of equivalence graphs.
We note that this result has a similar flavor to the transduction duality between star forests (which are transduction-equivalent to equivalence graphs) and classes of structurally bounded degree shown in \cite[Section 28.1]{braunfeld2025first}, though neither result follows readily from the other.

\begin{theorem}\label{thm:proper-equiv-struct-bdd-degree}
    Let $\cX$ be a class of graphs. Then $\cX$ has structurally bounded degree if and only if $\cX$ is a function of a \emph{proper} subclass of equivalence graphs.
\end{theorem}
\begin{proof}
    By \cref{th:clique-degree-1}, if $\cX$ is of structurally bounded degree, then it is a function of $\cC \cup \cD_1$, which is clearly a proper subclass of equivalence graphs. 
    
    To prove the converse, suppose that $\cX$ is a function of a proper subclass of equivalence graphs, denoted by $\cY$. Then, we claim that it is enough to show that $\cY$ is a function of $\cC \cup \cD_1$. Indeed, since $\mathcal{Y}$ is a function of  $\cC \cup \cD_1$, any graph that is a function of graphs in $\mathcal{Y}$ is also a function of graphs in $\cC \cup \cD_1$. Thus, $\cX$ is also a function of $\cC \cup \cD_1$, and therefore, by \cref{th:clique-degree-1}, $\cX$ has structurally bounded degree.

    To show $\cY$ is a function of $\cC \cup \cD_1$, we first observe that there exists a constant $k = k (\cY)$ such that every graph in $\cY$ has at most $k$ connected components of size more than $k$. Indeed, since $\cY$ is a proper subclass of equivalence graphs, there exists an equivalence graph $G$ that is not in $\cY$. Let $s$ be the number of connected components of $G$, and let $t$ be the size of its largest connected component.
    If some graph in $\cY$ had at least $s$ connected components, each of size at least $t$, then it would contain $G$ as an induced subgraph. Since $\cY$ is hereditary, this would imply $G\in\cY$, a contradiction. Thus, setting $k(\cY) := \max\{s,t\}-1$, every graph in $\cY$ has at most $k(\cY)$ connected components of size greater than $k(\cY)$.

    Now, let $G$ be a graph in $\cY$. It follows from the above that for some $s \leq k$, there exist $s$ graphs $F_1, F_2, \ldots, F_s \in \cC$ (i.e., equivalence graphs with at most one component of size more than 1) and a graph $D \in \cD_k$ (i.e., a graph of maximum degree at most $k$) such that $G = D \vee \bigvee_{i=1}^s F_i$. The graph $D$ covers all the components of $G$ of maximum degree at most $k$ while the graphs $F_i$ cover the at most $k$ components with degree greater than $k$. 
    Recall, from the proof of \cref{th:degree-1},  $D$ is a union of some $q \leq k+1$ graphs $M_1, M_2, \ldots, M_q \in \cD_1$.
    Thus, the graph $G = \bigvee_{i=1}^q M_i \vee \bigvee_{i=1}^s F_i$ is a function of $q+s$ graphs from $\cC \cup \cD_1$, implying that $\cY$ is a function of $\cC \cup \cD_1$. 
\end{proof}

\subsubsection{General equivalence graphs}

Let $G=(V,E)$ be an arbitrary graph. Note that $G$ can be obtained from the empty graph $(V, \emptyset)$ via a sequence of partition complementations; for example, every edge $(a,b)$ in $E$ can be created as a single partition complementation corresponding to the partition of $V$ consisting of one class $\{a,b\}$ and all other vertices being singleton classes.

The \emph{partition complementation} number of $G$ is the minimum number of partition complementation operations needed to construct $G$ from the empty graph. Formally, the partition complementation number of $G$ is the minimum $k$ such that there exists a sequence of graphs $G_0, G_1, G_2, \ldots, G_k$, where $G_0 = (V,\emptyset)$, $G_k = G$, and for every $i \in [k]$, $G_i$ is a partition complementation of $G_{i-1}$.

\begin{theorem}\label{th:general-equiv}
    Let $\cX$ be a class of graphs. 
    Then the partition complementation number is bounded in $\cX$ if and only if $\cX$ is a function of equivalence graphs.
\end{theorem}
\begin{proof}
    Suppose $\cX$ is a class of graphs and there exists $k \in \bN$ such that the partition complementation number of every graph in $\cX$ is at most $k$.
    
    Let $G=(V,E)$ be an arbitrary graph in $\cX$ and $G_0, G_1, G_2, \ldots, G_t$ be a sequence of graphs, where $t \leq k$, $G_0 = (V,\emptyset)$, $G_t = G$, and for every $i \in [t]$, $G_i$ is a partition complementation of $G_{i-1}$.
    By \cref{obs:xor-with-equivalence}, for every $i \in [t]$, there exists an equivalence graph $H_i$ such that $G_i = G_{i-1} \xor H_i$. Then $G = G_0 \xor H_1 \xor H_2 \xor \cdots \xor H_t = H_1 \xor H_2 \xor \cdots \xor H_t$. Thus, $G$ is a function of at most $k$ equivalence graphs, implying that $\cX$ is a function of the class of equivalence graphs.

    Conversely, suppose that $\cX$ is a $k$-function of the class of equivalence graphs for some $k \in \bN$. Since the latter class is intersection-closed and contains all complete graphs, by \cref{lem:intersection-closed-xor-representation}, every graph $G=(V,E)$ in $\cX$ is equal to $\Xor_{i=1}^s H_i$ for some $s \leq 2^k$ and some equivalence graphs $H_1, H_2, \ldots, H_s$.
    Let $H_0 = (V, \emptyset)$. Then, by \cref{obs:xor-with-equivalence}, for every $i \in [s]$, $H_1 \xor H_2 \xor \cdots \xor H_i$ is a partition complementation of 
    $H_0 \xor H_1 \xor H_2 \xor \cdots \xor H_{i-1}$. Consequently, $G$ is obtained from the empty graph $H_0$ via a sequence of at most $2^k$ partition complementations, implying that the partition complementation number is bounded in $\cX$.
\end{proof}

\section{Boolean functions and \texorpdfstring{$\chi$}{chi}-boundedness}
\label{sec:chi-boundedness}

In this section we explore how Boolean combinations of graph classes may affect $\chi$-boundedness.
We begin in \cref{sec:intbdd} with some general results.
In \cref{sec:interval-graphs}, we consider interval, split, and some related graph classes. Classes of interval and split graphs are both subclasses of perfect graphs. It is known that not every function of interval graphs is $\chi$-bounded (see \cref{thm:interval-intersection}). In contrast, we show that any function of the class of split graphs is $\chi$-bounded.
Similarly, permutation graphs are perfect, and we show in \cref{sec:permutation} that any function of the class of permutation graphs is \emph{polynomially} $\chi$-bounded. This, in particular, implies that any function of the class of cographs or the class of equivalence graphs is polynomially $\chi$-bounded.
In \cref{sec:chi-boundedness-equiv-graphs}, we consider in more detail the class of equivalence graphs as one of the simplest subclasses of perfect graphs. First, we show that for any polynomial $p(x)$ there exists a function of equivalence graphs whose smallest $\chi$-binding function is lower bounded by $p(x)$. Then, we identify some restrictions under which functions of equivalence graphs are \emph{linearly} $\chi$-bounded. Finally, we show that in some cases functions of equivalence graphs preserve perfectness. 

Since the union of finitely many (linearly/polynomially) $\chi$-bounded classes is (linearly/polynomially) $\chi$-bounded, for our results in this section we usually show (linear/polynomial) $\chi$-boundedness of $f(\cX)$ for an arbitrary, but fixed function $f$ (see \cref{rem:single-function}).

\subsection{General classes}\label{sec:intbdd}

We begin with a result from \cite{Gya87}, showing that a disjunction of $\chi$-bounded classes is also $\chi$-bounded.

\begin{lemma}[{\cite[Proposition 5.1(a)]{Gya87}}]
\label{lem:union}
    Let $\cX_1, \cX_2, \ldots, \cX_k$ be $\chi$-bounded classes of graphs with
    $\chi$-binding functions $h_1(x), h_2(x), \ldots, h_k(x)$ respectively. Then the class $\bigvee_{i=1}^k \cX_i$ is $\chi$-bounded and $g(x)=\prod_{i=1}^k h_i(x)$ is a suitable $\chi$-binding function. In particular, for any $t \in \bN$, the $t$-union of a $\chi$-bounded class is also $\chi$-bounded.
\end{lemma}

In the next statements, we show $\chi$-boundedness for more general functions when some restrictions are imposed on graph classes.

\begin{lemma} \label{lem:intersection}
Let $t \in \bN$, and  $f : \zo^t \rightarrow \zo$ be a monotone Boolean function. If $\cX$ is a $\chi$-bounded class such that the $t$-intersection of $\cX$ is $\chi$-bounded, then $f(\cX)$ is $\chi$-bounded.
Furthermore, if $\cX$ is intersection-closed (resp.\ monotone) and for every $s \in \bN$, $h_s(x)$ denotes a $\chi$-binding function for the $s$-union of $\cX$, then $h_{2^t-1}(x)$ (resp.\ $h_t(x)$) is a $\chi$-binding function for $f(\cX)$. 
\end{lemma}
\begin{proof}
Note that if $f\equiv 1$ (i.e., $f$ is a constant 1 function), then $f(\cX)$ is contained in the class of complete graphs. Therefore, $f(\cX)$ is a class of perfect graphs, and thus it is $\chi$-bounded with the identity binding function, and the lemma obviously holds in this case. 
In the rest of the proof we assume that $f \not\equiv 1$.

Let $\cY=f(\cX)$. Since $f$ is monotone, by \cref{th:monotone-DNF}, it can be represented by a monotone DNF $F = \bigvee_{i=1}^r D_i$, where $r \in [2^t-1]$, $D_i = \bigwedge_{j=1}^{c_i} x_{a_j^i}$, and $c_i \in [t]$. 
Thus, any graph in $\cY$ can be expressed as $\bigvee_{i=1}^r \bigwedge_{j=1}^{c_i} G_{a_j^i}$ for some graphs $G_{a_j^i} \in \cX$. Denoting by $\cS_{k}$ the $k$-intersection of $\cX$, we conclude that $\cY \subseteq\bigvee_{i=1}^r \cS_{c_i}$. 

Now, since $\cS_k \subseteq \cS_t$ for every $k \in [t]$, and, by assumption, $\cS_t$ is $\chi$-bounded, we have that each $\cS_k$, $k \in [t]$, is $\chi$-bounded. Thus, $\cY$ is also $\chi$-bounded by \cref{lem:union}.

Next, suppose that $\cX$ is intersection-closed. Then, for every $k \in [t]$, we have $\cX = \cS_1 \subseteq \cS_k \subseteq \cX$, and thus $\cS_k = \cX$. Therefore, $\cY$ is contained in the $2^t-1$-union of $\cX$. Thus $h_{2^t-1}(x)$ is a $\chi$-binding function for $\cY$. 

Assume now that $\cX$ is monotone. Let $\Gamma_1$  be the set of conjunctions of $F$ that contain variable $x_1$, and for $2 \leq j \leq t$, define $\Gamma_j$ as the set of conjunctions of $F$ that contain variable $x_j$, but none of the variables $x_1, x_2, \ldots, x_{j-1}$. Then, by regrouping the conjunctions, $F$ can be written as
\[
    F = \bigvee_{i=1}^t \bigvee_{D \in \Gamma_i} D.
\]
In turn, $\bigvee_{D \in \Gamma_i} D$ can be written as 
$x_i \wedge \left(\bigvee_{D \in \Gamma_i} D^{x_i}\right)$, 
where $D^{x_i}$ is the conjunction obtained from $D$ by removing variable $x_i$.
Thus, if $H_1, H_2, \ldots, H_t$ are some graphs from $\cX$, then 
\[
    f(H_1, H_2, \ldots, H_t) = \bigvee_{i=1}^t \left[ H_i \wedge F_i \right],
\]
for some graphs $F_1, F_2, \ldots, F_t$.
Note that $H_i \wedge F_i$ is a subgraph of $H_i$ and thus belongs to $\cX$, due to the monotonicity of $\cX$. Consequently, $f(H_1, H_2, \ldots, H_t)$ belongs to the $t$-union of $\cX$, which implies that $\cY \subseteq \bigvee_{i=1}^t \cX$, and hence $h_t(x)$ is a $\chi$-binding function for $\cY$. 
\end{proof}

\begin{lemma} \label{lem:complement}
    Let $\cX$ be a $\chi$-bounded  class, and let $t$ be a natural number.
    For integers $k,\ell\geq 1$, denote by $\cS_k$ the $k$-intersection of $\cX$, by $\cT_{\ell}$ the $\ell$-intersection of $\overline{\cX}$, and let $\cS_0$ and $\cT_0$ denote the class of complete graphs. Then all $t$-functions of $\cX$ are $\chi$-bounded if and only if $\cS_k\wedge\cT_{\ell}$ is $\chi$-bounded for all integers $k, \ell \geq 0$ with $k+\ell = t$.
\end{lemma}
\begin{proof}
    Assume $\cS_k\wedge\cT_\ell$ is $\chi$-bounded for any integers $k \geq 0$ and $\ell \geq 0$ such that $k+\ell= t$, and let $f$ be a Boolean function on $t$ variables. Then $f$ can be expressed as a DNF with at most $2^t$ conjunctions, where each conjunction consists of $t$ literals. Therefore, each graph in $f(\cX)$ is the union of at most $2^t$ graphs from classes $\{\cS_k\wedge\cT_{\ell}\}_{k,\ell\geq 0\text{, }k+\ell = t}$. As each $\cS_k\wedge\cT_{\ell}$ is $\chi$-bounded we have that $f(\cX)$ is $\chi$-bounded by Lemma \ref{lem:union}. 

    The other direction holds trivially as each of the classes  $\{\cS_k\wedge\cT_{\ell}\}_{k,\ell\geq 0\text{, }k+\ell = t}$ is a $t$-function of $\cX$ by definition.
\end{proof}

\begin{lemma}\label{lem:all}
Let $t \in \bN$ and let $\cX$ be a \emph{self-complementary} $\chi$-bounded class.
If the $t$-intersection of $\cX$ is $\chi$-bounded with a $\chi$-binding function $h(x)$, then any $t$-function of $\cX$ is $\chi$-bounded with a $\chi$-binding function $(h(x))^{2^t}$.
\end{lemma}
\begin{proof}
    Using notation from \cref{lem:complement}, since $\cX=\overline{\cX}$, for each $k, \ell \geq 0$ with $k+\ell = t$, we have that $\cS_k\wedge\cT_\ell = \cS_{k+\ell} = \cS_{t}$.
    Thus, if $f$ is a Boolean function on at most $t$ variables, then each graph in $f(\cX)$ is a union of at most $2^t$ graphs from the class $\cS_{t}$. Consequently, by \cref{lem:union}, $f(\cX)$ is $\chi$-bounded with a $\chi$-binding function $(h(x))^{2^t}$.
\end{proof}

\subsection{Interval, split, and related graph classes}
\label{sec:interval-graphs}

Interval graphs are intersection graphs of intervals on the real line.
Although interval graphs are perfect, it is known that the class of $3$-intersections of interval graphs is not $\chi$-bounded \cite{Gya87}. 
\begin{theorem}[{\cite[Theorem 5.6 (b)]{Gya87}}]
\label{thm:interval-intersection}
    The $t$-intersection of interval graphs is not $\chi$-bounded for any $t\geq 3$.
\end{theorem}
Note that this result demonstrates that complementation does not always preserve $\chi$-boundedness. Indeed, denote by $\cI$ the class of interval graphs. Since interval graphs are perfect, the class $\overline{\cI}$ is a subclass of perfect graphs and thus $\chi$-bounded. Therefore, by \cref{lem:union}, $\cX = \overline{\cI} \vee \overline{\cI} \vee \overline{\cI}$ is $\chi$-bounded, but $\overline{\cX} = \cI \wedge \cI \wedge \cI$ is not by \cref{thm:interval-intersection}.

In this section, we show that, in contrast to \cref{thm:interval-intersection}, for some classes related to interval graphs, intersections preserve $\chi$-boundedness.

We start with the class of split graphs, which are graphs whose vertex set can be partitioned into an independent set and a clique. We show that for any fixed $t \in \bN$, the $t$-intersection of the class of split graphs is linearly $\chi$-bounded. As a consequence, we show that any function of the class of split graphs is $\chi$-bounded.
This is interesting as almost all chordal graphs are split graphs \cite{BRW85}, yet the $3$-intersection of chordal graphs is not $\chi$-bounded, as interval graphs are chordal graphs.

\begin{theorem}\label{lem:new-split}
        For any $t \in \bN$, the $t$-intersection of the class of split graphs is linearly $\chi$-bounded with a $\chi$-binding function $h(x) = x+t$.
\end{theorem}
\begin{proof}
     Let $\cX$ be the $t$-intersection of the class of split graphs.
    We will show that each graph in $\cX$ can be partitioned into a clique and at most $t$ independent sets.

    Let $G=(V,E)\in\cX$. Then there exist $t$ split graphs, $H_1,\ldots,H_t,$ such that $G=\bigwedge_{i=1}^t H_i$ and each $H_i$ has a fixed partition of its vertices into an independent set $S_i$ and a clique $Q_i$. As $Q_i$ is a clique in $H_i$ for every $i\in[t]$, it follows that $\bigcap_{i=1}^t Q_i$ is a clique in $G$, and thus it can be colored with at most $\omega(G)$ colors. Note that for every vertex $v\in V\backslash {\bigcap_{i=1}^t Q_i}$, there is an $i\in [t]$ such that $v\in S_i$, and thus 
    $V\backslash {\bigcap_{i=1}^t Q_i}=\bigcup_{i=1}^t S_i$. As any independent set in $H_i$ is an independent set in $G$ for any $i\in[t]$, the set $V\backslash {\bigcap_{i=1}^t Q_i}$ can be covered with $t$ independent sets $S_1, S_2, \ldots, S_t$ of $G$. This implies that $G$ can be colored with at most $\omega(G)+t$ colors. 
\end{proof}

\begin{theorem}\label{lem:split}
    For any $t \in \bN$, any $t$-function of the class of split graphs is polynomially $\chi$-bounded with a $\chi$-binding function $g(x)=(x+t)^{2^t}$.
\end{theorem}
\begin{proof}
   As the class of split graphs is self-complementary and, by \cref{lem:new-split}, for any $t \in \bN$, the $t$-intersection of split graphs is $\chi$-bounded with a $\chi$-binding function $h(x) = x+t$, it follows from \cref{lem:all} that any $t$-function of the class of split graphs is $\chi$-bounded and $g(x) = (h(x))^{2^t} = (x+t)^{2^t}$ is a suitable $\chi$-binding function. 
\end{proof}

Next, we consider the class of $\{ K_{1,3}, C_4, C_5 \}$-free graphs, which includes the class of unit interval graphs, and, more generally, the class of $K_{1,3}$-free chordal graphs. We show that any $t$-intersection of this class is polynomially $\chi$-bounded, and then extend this to the class of all $K_{1,p}$-free graphs.

The Ramsey number $R(p,q)$ is the minimum $n$ such that any red-blue edge coloring of the complete graph $K_n$ contains a red $K_p$, or a blue  $K_q$. 
It is known that for any $k \geq 3$, there exists $c_k$ such that $R(x,k)\leq c_k\frac{x^{k-1}}{(\ln x)^{k-2}}$ \cite{ajtai1980note}. In particular, when $k$ is fixed, $R(x,k)$ grows at most polynomially with respect to $x$.
More generally, the multicolor Ramsey number $R(p_1,p_2,\ldots,p_t)$ is the minimum $n$ such that in any edge coloring of $K_n$ with $t$ colors, there exists a monochromatic clique of size $p_i$ for some $i \in [t]$. If $p=p_1=p_2=\ldots=p_t$, then we write $R(p_1,p_2,\ldots,p_t)=R_t(p)$.

\begin{lemma}\label{unit interval}
For any $t \in \bN$, the $t$-intersection of the class of $\{ K_{1,3}, C_4, C_5 \}$-free graphs is $\chi$-bounded with a \emph{polynomial} $\chi$-binding function $h(x) = R(x,2^t+1)$.  
\end{lemma}
\begin{proof}
    Let $\cX$ be the class of $\{ K_{1,3}, C_4, C_5 \}$-free graphs. 
    Note that since $\cX$ is a hereditary class, the $t$-intersection of $\cX$ is also a hereditary class.
    To prove the statement we will show that $K_{1,2^t+1}$ does not belong to the $t$-intersection of $\cX$, i.e., it is a forbidden induced subgraph for the class. The result then follows from the fact that, for any $p \geq 2$, the class of $K_{1,p}$-free graphs is $\chi$-bounded with a $\chi$-binding function $g(x) = R(x,p)$ \cite[Theorem 2.2]{Gya87}.
    
    Assume that $K_{1,2^t+1}$ belongs to the $t$-intersection of $\cX$. Then there exist $t$ graphs in $\cX$, denoted $G_1,\ldots,G_t$, such that $K_{1,2^t+1}=\bigwedge_{i=1}^t G_i$. So $K_{2^t+1}+ O_1$ (the complement of $K_{1,2^t+1}$) is equal to $\bigvee_{i=1}^t\overline{G_i}$. 
    Note that for every $i \in [t]$, graph $\overline{G_i}$ is $\{ K_3+ O_1, 2K_2, C_5 \}$-free, as $G_i$ is $\{ K_{1,3}, C_4, C_5 \}$-free. 
    
    Let $w$ be the isolated vertex of the graph $K_{2^t+1}+ O_1$. Then, $w$ is an isolated vertex in every $\overline{G_i}$, $i \in [t]$, since each of these graphs is ($K_3+ O_1$)-free, we conclude that each $\overline{G_i}$ is $K_3$-free. Consequently, each $\overline{G_i}$, $i \in [t]$, is $\{ K_3, 2K_2, C_5 \}$-free, and, in particular, bipartite.
    However, from Lemma \ref{lem:union}, \[2^t+1=\chi(K_{2^t+1}+ O_1)=\chi(\overline{G_1}\vee\ldots\vee \overline{G_t}) \leq \prod_{i=1}^t \chi(\overline{G_i})=2^t,\] which is clearly a contradiction. 
\end{proof}

We now generalize the above result to the class of $K_{1,p}$-free graphs at the cost of a worse bound on the $\chi$-binding function.

\begin{lemma}\label{claw free}
    For any $t,p \in \bN$, the $t$-intersection of the class of $K_{1,p}$-free graphs is $\chi$-bounded with a \emph{polynomial} $\chi$-binding function $h(x) = R(x,R_t(p))$.
\end{lemma}
\begin{proof}
    Let $\cX$ be the class of $K_{1,p}$-free graphs. 
    As before, to prove the statement we will show that
    $K_{1,R_t(p)}$ is a forbidden induced subgraph for the $t$-intersection of $\cX$, which together with \cite[Theorem 2.2]{Gya87} would immediately imply the desired result.
    
    Assume that there exist $t$ graphs $G_1,\ldots,G_t \in \cX$ such that $K_{1,R_t(p)}=\bigwedge_{i=1}^t G_i$. Then $K_{R_t(p)}+ O_1$ (the complement of $K_{1,R_t(p)}$) is equal to $\bigvee_{i=1}^t\overline{G_i}$. Note that for every $i \in [t]$, graph $\overline{G_i}$ is $(K_p + O_1)$-free, as $G_i$ is $K_{1,p}$-free. 

    Let $w$ be the isolated vertex of the graph $K_{R_t(p)}+ O_1$. Then, $w$ is isolated in every $\overline{G_i}$, $i \in [t]$, and since  each of these graphs is ($K_p+ O_1$)-free, we conclude that each $\overline{G_i}$ is $K_p$-free. 
    On the other hand, the union of $\overline{G_i}$, $i \in [t]$, contains $K_{R_t(p)}$ as an induced subgraph. Without loss of generality, we can assume that the graphs $\overline{G_i}$, $i \in [t]$, are edge-disjoint as if two of them have an edge in common we can remove the edge from one of them without changing their union or creating a $K_p$. 
    This edge-disjoint union of $t$ many $K_p$-free graphs 
    corresponds to a coloring of the edges of the $K_{R_t(p)}$ with $t$ colors that does not contain a monochromatic $K_p$, contradicting the definition of $R_t(p)$.
\end{proof}

\subsection{Permutation graphs}\label{sec:permutation}

The comparability graph of a partial order $(V,<)$ is the graph with vertex set $V$ in which two elements are adjacent if and only if they are comparable in the partial order. A graph is called a \emph{comparability graph} if it is a comparability graph of some partial order. A graph is \emph{co-comparability} if it is the complement of a comparability graph.

It is easy to see that any interval graph is a co-comparability graph, where the corresponding partial order can be naturally defined on the intervals of a geometric representation of the interval graph.
This together with \cref{thm:interval-intersection} implies that the 3-intersection of the class of co-comparability graphs is not $\chi$-bounded, and \emph{not} every function of the class of comparability graphs is $\chi$-bounded. However, Gy{\'a}rf{\'a}s showed \cite{Gya87} that the $t$-intersection of the class of comparability graphs is polynomially $\chi$-bounded for every $t \in \bN$. 

\begin{lemma}[{\cite[Proposition 5.8]{Gya87}}]\label{lem:comparability}
    For any $t \in \bN$,
    the $t$-intersection of the class of comparability graphs is polynomially $\chi$-bounded with a $\chi$-binding function $h(x) = x^{2^{t-1}}$.
\end{lemma}

Graphs that are both comparability and co-comparability are precisely the permutation graphs. The next result shows that any function of the class of permutation graphs is polynomially $\chi$-bounded.
We note that this is a more general formulation of a result of Gy\'arf\'as. We include the proof for completeness.
\begin{theorem}[\cite{Gya87}, Proposition 5.10]\label{th:fun-permutation}
    Any function of the class of permutation graphs is polynomially $\chi$-bounded.
\end{theorem}
\begin{proof}
Let $t \in \bN$, and $\cX$ be the class of permutation graphs. 
Then $\cX=\overline{\cX}$, as the complement of a permutation graph is also a permutation graph. Additionally, $\cX$ is contained in the class of comparability graphs, so, by \cref{lem:comparability}, the $t$-intersection of $\cX$ is $\chi$-bounded with $h(x) = x^{2^{t-1}}$ as a $\chi$-binding function. By Lemma \ref{lem:all}, this implies that any $t$-function of $\cX$ is $\chi$-bounded with a $\chi$-binding function $g(x)=(x^{2^{t-1}})^{2^t}=x^{2^{2t-1}}$.  
\end{proof}

We note that the fact that every Boolean function of a graph class $\cX$ is $\chi$-bounded yields separation results. In particular, if a graph class admits a Boolean function that is not $\chi$-bounded, then it cannot itself be a function of $\cX$.

For example, the class of interval graphs is not a function of the class of permutation graphs, nor of the class of split graphs. Indeed, Boolean functions are closed under composition. Hence, if interval graphs could be expressed as a function of permutation graphs (resp., split graphs), then \cref{th:fun-permutation} (resp., \cref{lem:split}) would imply that every Boolean function of interval graphs is $\chi$-bounded, contradicting \cref{thm:interval-intersection}. This generalises the result of \cite{KratochvilT94}, which shows that there is no finite $k$ such that every interval graph can be expressed as a $k$-intersection of permutation graphs (resp., split graphs).
Similarly, it is known that some Boolean functions of the class of chair-free graphs, as well as some Boolean functions of the class of $(K_{1,3}+O_1)$-free graphs, are not $\chi$-bounded \cite{CKSintersection}. Consequently, neither of these classes can be a function of the class of permutation graphs or of the class of split graphs.

Similar separation results can also be obtained for restricted families of Boolean functions. For example, by \cref{lem:intersection} and \cref{lem:comparability}, every monotone Boolean function of the class of comparability graphs is $\chi$-bounded. Thus, none of the classes above can be a monotone function of the class of comparability graphs. For interval graphs, this strengthens a result of \cite{GMT24}, which shows that there is no finite $k$ such that every interval graph is a union of $k$ comparability graphs.

More generally, graph properties other than $\chi$-boundedness can be used to obtain analogous separation results; we refer the reader to \cref{sec:conclusiong} for a broader discussion.

To conclude, we mention two special cases of \cref{th:fun-permutation}.
As cographs and equivalence graphs are permutation graphs, we have the following corollaries.
\begin{corollary}\label{th:fun-cographs}
    Any function of the class of cographs is polynomially $\chi$-bounded.
\end{corollary}

\begin{corollary}\label{th:fun-equivalence}
    Any function of the class of equivalence graphs is polynomially $\chi$-bounded.
\end{corollary}

Note that equivalence graphs form one of the simplest classes of perfect graphs.
In the next section we study the $\chi$-boundedness of functions of this class in more detail.

\subsection{Equivalence graphs}\label{sec:chi-boundedness-equiv-graphs}

Functions of equivalence graphs are polynomially $\chi$-bounded due to \cref{th:fun-equivalence}.
In \cref{sec:chi-non-linear}, we show that in general this cannot be improved. 
More specifically, we prove that for any given polynomial there exists a function of equivalence graphs whose smallest $\chi$-binding function grows at least as fast as this polynomial.
In \cref{sec:chi-linear}, we complement this result by identifying some restrictions under which functions of equivalence graphs are linearly $\chi$-bounded. In particular, we show that any function of a proper subclass of equivalence graphs and any monotone function of the class of equivalence graphs are linearly $\chi$-bounded.
Finally, in \cref{sec:chi-perfect}, we show that all functions of two equivalence graphs and some (but not all) functions of three equivalence graphs are subclasses of perfect graphs.

\subsubsection{Strictly polynomially \texorpdfstring{$\chi$}{chi}-bounded classes}\label{sec:chi-non-linear}

In this section, we show that for any $s \in \bN$, there exists a function of the class of equivalence graphs whose smallest $\chi$-binding function is lower bounded by a polynomial of degree $s$.

For $k \in \bN$, we denote by $\cB_k$ the $k$-XOR of the class of equivalence graphs. Let $H(n,k)$ be the graph with the vertex set
$[n]^k$, where two vertices $v=(v_1,v_2,\ldots,v_k)$ and $w=(w_1,w_2,\ldots,w_k)$ are adjacent if and only if they agree on an odd number of coordinates, i.e., $|\{ i \in [k] ~|~ v_i=w_i \}| \equiv 1 \pmod{2}$. We claim that for any $n \in \bN$, the graph $H(n,k)$ belongs to $\cB_k$. Indeed, denote by $G_i$ the equivalence graph with vertex set $[n]^k$
and edge set $\{ (v,w) ~|~ v,w \in [n]^k, v \neq w, v_i = w_i \}$. Then, from \cref{obs:xor-of-k-graphs}, we conclude that $H(n,k) = G_1 \xor G_2 \xor \cdots \xor G_k$.

We will show (\cref{th:binding lower bound}) that for any $s \geq 2$, the smallest $\chi$-binding function for $\cB_{2s}$ is $\Omega_s(x^s)$, where subscript $s$ in $\Omega$ indicates that the hidden constant depends on $s$. To do so, we will establish upper bounds (\cref{lem:Hnk-clique-independent-set-even}, \cref{lem:Hnk-clique-independent-set-odd}) on the clique and independence numbers of $H(n,k)$, from which we will derive a suitable lower bound on the chromatic number in terms of the clique number. We now recall some classical results on set systems with restricted intersections.

\begin{theorem}[Oddtown Theorem {\cite[Corollary 1.2]{babai2022linear}}]\label{Oddtown}
    Let $\cF \subseteq 2^{[n]}$ be such that $|A|$ is odd for all $A\in\cF$, and $|A\cap B|$ is even for all distinct $A,B\in\cF$. Then $|\cF|\leq n$.
\end{theorem}

\begin{theorem}[Reverse Oddtown Theorem {\cite[Exercise 1.15]{babai2022linear}}]\label{Reverse Oddtown}
     Let $\cF\subseteq 2^{[n]}$ be such that $|A|$ is even for all $A\in\cF$, and $|A\cap B|$ is odd for all distinct $A,B\in\cF$. Then $|\cF|\leq n$.
\end{theorem}

Let $L$ be a set of integers. A set family $\cF \subseteq 2^{[n]}$ is \emph{$L$-intersecting} if $|A\cap B|\in L$ for any distinct $A,B\in \cF$; and $\cF$ is 
\emph{$k$-uniform} if $|A|=k$ for all $A\in\cF$.

\begin{theorem}[Ray-Chaudhuri-Wilson Theorem~\cite{RCW}]\label{Uniform Ray-Chaudhuri-Wilson}
    Let $n,k \in \bN$ and $k \leq n$.
    Let $L \subseteq \{0\}\cup [n]$ be a set of integers and $\cF\subseteq 2^{[n]}$ an $L$-intersecting $k$-uniform family of sets. Then, $|\cF|\leq \binom{n}{|L|}$. 
\end{theorem}
We will now apply these three results to bound the clique and independence numbers of $H(n,k)$. 
\begin{lemma}\label{lem:Hnk-clique-independent-set-even}
    Let $n, k \in \bN$, and $k$ be even. Then,
    \begin{enumerate}
        \item[(1)] the clique number of $H(n,k)$ is at most $nk$,
        \item[(2)] the independence number of $H(n,k)$ is at most $(2en)^{\frac{k}{2}}$.
    \end{enumerate}
\end{lemma}
\begin{proof}
    We start with the clique number. By definition, a clique in $H(n,k)$ corresponds to a set of vectors in $[n]^k$ such that any two vectors agree on an odd number of entries. 
    Consider the set $S$ obtained by adding the vector $(0,n,2n, \ldots, (k-1)n)$ to every vector in $[n]^k$, i.e., $S := \{ v + (0,n,2n, \ldots, (k-1)n) ~|~ v \in [n]^k \}$. Note that this transformation does not change whether two vectors agree on a particular entry. Furthermore, for every vector $u \in S$ and every $t \in [k]$, the $t$-th entry of $u$ is in the interval $[(t-1)n+1, tn]$.
    Thus, there is a one-to-one correspondence between $S$ (and hence $[n]^k$) and the family $\cF$ of subsets of size $k$ of $[nk]$ which have exactly one element in $[(t-1)n+1, tn]$ for $t \in [k]$. This implies a one-to-one correspondence between subsets of $[n]^k$ whose elements pairwise agree on an odd number of entries and subfamilies of $\cF$ in which any two sets intersect at an odd number of elements.
    Since $k$ is even, by \cref{Reverse Oddtown}, any such family has at most $nk$ sets, which implies (1).

    Consider now the independence number. By definition, any independent set in $H(n,k)$ corresponds to a set of vectors in $[n]^k$ such that any two vectors agree on an even number of entries. As in the previous case, we can establish 
    a one-to-one correspondence between such sets of vectors and subfamilies of $\cF\subseteq 2^{[nk]}$ in which any two sets intersect at an even number of elements, i.e., $k$-uniform subfamilies of $\cF$ that are $L$-intersecting with $L=\{2i ~|~ 0\leq i \leq k/2-1\}$ (note that $k/2$ is excluded from the range of $i$ as the intersection of two distinct subsets of size $k$ cannot be $k$).
    By \cref{Uniform Ray-Chaudhuri-Wilson}, the largest possible size of such a family is at most 
    \[
        \binom{nk}{|L|} = \binom{nk}{\frac{k}{2}} \leq \frac{(nk)^{\frac{k}{2}}}{(k/2)!}\leq (2en)^{\frac{k}{2}},
    \]
    which implies (2).
\end{proof}

With a very similar argument, using \cref{Oddtown} instead of \cref{Reverse Oddtown}, one can show the following lemma; we omit the proof.

\begin{lemma}\label{lem:Hnk-clique-independent-set-odd}
    Let $n, k \in \bN$, and $k$ be odd. Then,
    \begin{enumerate}
        \item[(1)] the clique number of $H(n,k)$ is at most $(2en)^{\frac{k-1}{2}}$, 
        \item[(2)] the independence number of $H(n,k)$ is at most $nk$.
    \end{enumerate}
\end{lemma}

We are now ready to prove the main result of this section. 

\begin{theorem}\label{th:binding lower bound}
    Let $k \geq 3$ be an integer. Then the smallest $\chi$-binding function for $\cB_k$ is $\Omega_k(x^{\lfloor k/2 \rfloor})$ if $k \geq 4$, and is $\Omega(x^2)$ if $k=3$.
\end{theorem}
\begin{proof}
    First, let $k \geq 4$ be even. Denote by $G=(V,E)$ the graph $H(n,k)$.
    Since $\chi(G)\alpha(G)\geq |V|$, by \cref{lem:Hnk-clique-independent-set-even}(2), we have 
    $\chi(G)(2en)^{\frac{k}{2}}\geq n^k$, and thus from \cref{lem:Hnk-clique-independent-set-even}(1) we obtain
    \[
        \chi(G)
        \geq \frac{n^{\frac{k}{2}}}{(2e)^{\frac{k}{2}}}
        = \frac{(nk)^{\frac{k}{2}}}{(2ek)^{\frac{k}{2}}}
        \geq \frac{\omega(G)^{\lfloor k/2 \rfloor}}{(2ek)^{\frac{k}{2}}}.
    \]

    To obtain the bound for every sufficiently large clique number, let
    $t\in [kn,k(n+1)-1]$
    and let $G'=H(n,k)+K_t$. Then $G'\in\cB_k$. Indeed, starting with a representation of $H(n,k)$ as the $k$-XOR of equivalence graphs, we may add a disjoint copy of $K_t$ to one of these graphs and $t$ isolated vertices to each of the remaining graphs. The resulting graphs are still equivalence graphs, and their $k$-XOR is $G'$.

    Since $\omega(H(n,k))\leq nk\leq t$, we have $\omega(G')=t$. Moreover,
    $\alpha(G') \leq (2en)^{\frac{k}{2}}+1<(3en)^{\frac{k}{2}}$. 
    Since $t\leq k(n+1)-1\leq 2nk$, it follows that
    \[
        \chi(G')
        \geq \frac{|V(G')|}{\alpha(G')}
        \geq \frac{n^{\frac{k}{2}}}{(3e)^{\frac{k}{2}}}
        =\frac{(2nk)^{\frac{k}{2}}}{(6ek)^{\frac{k}{2}}}
        \geq
        \frac{\omega(G')^{\lfloor k/2\rfloor}}
        {(6ek)^{\frac{k}{2}}}.
    \]
    As the intervals $[kn,k(n+1)-1]$, over all $n\in\bN$, contain every sufficiently large integer, this proves the desired lower bound when $k$ is even.

    Now, let $k\geq 5$ be odd. Let $t\in[(k-1)n,(k-1)(n+1)-1]$ and denote by $G=(V,E)$ the graph
    \[
        (H(n,k-1)+K_t)\xor O_{n^{k-1}+t}
        =H(n,k-1)+K_t,
    \]
    where, recall, $O_m$ denotes the empty graph on $m$ vertices.
    By the previous case, $H(n,k-1)+K_t$ belongs to $\cB_{k-1}$, and hence $G$ belongs to $\cB_k$. Furthermore, since $k-1\geq 4$ and is even, by the previous case we conclude
    \[
        \chi(G)
        \geq
        \frac{\omega(G)^{\frac{k-1}{2}}}
        {(6e(k-1))^{\frac{k-1}{2}}}
        =
        \frac{\omega(G)^{\lfloor k/2\rfloor}}
        {(6e(k-1))^{\frac{k-1}{2}}}.
    \]
    Since the intervals $[(k-1)n,(k-1)(n+1)-1]$ contain every sufficiently large integer, the desired lower bound follows for odd $k\geq 5$.

    Finally, let $k=3$. Denote by $G=(V,E)$ the graph $H(n,3)$. Then, by \cref{lem:Hnk-clique-independent-set-odd}(2), we have $\chi(G)\cdot 3n \geq n^3$, and thus from \cref{lem:Hnk-clique-independent-set-odd}(1) we obtain
    \[
        \chi(G)
        \geq \frac{n^2}{3}
        =\frac{(2en)^2}{3(2e)^2}
        \geq\frac{\omega(G)^2}{3(2e)^2}.
    \]
    As before, to prove the bound for  every sufficiently large clique number, let
    $t \in [\lceil 2en\rceil,\lceil 2e(n+1)\rceil-1]$
    and let $G'=H(n,3)+K_t \in\cB_3$. Since
    $\omega(H(n,3))\leq 2en\leq t$, we have $\omega(G')=t$. Moreover, $\alpha(G')\leq 3n+1\leq 4n$.
    Since
    $t\leq\lceil 2e(n+1)\rceil-1<2e(n+1)\leq 4en$,
    it follows that
    \[
        \chi(G')
        \geq\frac{|V(G')|}{\alpha(G')}
        \geq\frac{n^2}{4}
        =\frac{(4en)^2}{4(4e)^2}
        \geq\frac{\omega(G')^2}{4(4e)^2}.
    \]
    Since the intervals
    $[\lceil 2en\rceil,\lceil 2e(n+1)\rceil-1]$,
    over all $n\in\bN$, contain every sufficiently large integer, we conclude that the smallest $\chi$-binding function for $\cB_3$ is $\Omega(x^2)$.
\end{proof}

\subsubsection{Linearly \texorpdfstring{$\chi$}{chi}-bounded classes}\label{sec:chi-linear}

In this section we reveal special cases in which functions of the class of equivalence graphs are linearly $\chi$-bounded.

We start with arbitrary functions of proper subclasses of equivalence graphs.
Recall, by \cref{thm:proper-equiv-struct-bdd-degree}, any such class has structurally bounded degree, and thus structurally bounded expansion \cite{GKNOdMPST20}. Graph classes of structurally bounded expansion are known to be linearly $\chi$-bounded \cite{NOdMPRS21}. Thus, we have the following.

\begin{theorem}\label{thm:proper-equiv-linearly-chi}
    Let $\cX$ be a \emph{proper} subclass of equivalence graphs.
    Then any function of $\cX$ is linearly $\chi$-bounded.
\end{theorem}

Next we consider the entire class of equivalence graphs, but restrict the Boolean functions.
Recall, \cref{lem:union} says that the union of $\chi$-bounded classes is also $\chi$-bounded, where the product of the corresponding $\chi$-binding functions is a $\chi$-binding function for the union class.
We show that in the case of equivalence graphs, the product bound on the $\chi$-binding function for the union class can be replaced by a sum. 

\begin{lemma}\label{lem:equiv}
   Let $t \in \bN$. The $t$-union of the class of equivalence graphs is linearly $\chi$-bounded with a $\chi$-binding function $h(x) = tx$.
\end{lemma}
\begin{proof}
    Let $H=G_1 \vee \ldots \vee G_t$, where $G_i$, $i \in [t]$, are equivalence graphs. Observe that, for every $i \in [t]$, $\Delta(G_i) = \omega(G_i)-1$, and $\omega(G_i) \leq \omega(H)$. Furthermore, 
    $\Delta(H) \leq \sum_{i=1}^t \Delta(G_i)$. From these, and the fact that the chromatic number of a graph does not exceed its maximum degree by more than 1, we conclude the desired result:
    \[
        \chi(H) 
        \leq \Delta(H) + 1 
        \leq \sum_{i=1}^t \Delta(G_i) + 1 
        = \sum_{i=1}^t (\omega(G_i)-1) + 1 \leq t(\omega(H)-1)+1 \leq t \omega(H).\qedhere
    \]
\end{proof}

As a corollary of \cref{lem:equiv} and \cref{lem:intersection} we obtain that any monotone function of the class of equivalence graphs is linearly $\chi$-bounded.

\begin{theorem}
    Let $t \in \bN$, and let $\cX$ be the class of equivalence graphs. For any monotone Boolean function $f$ on $t$ variables, the graph class $f(\cX)$ is linearly $\chi$-bounded with a $\chi$-binding function $h(x)=(2^t-1)x$. 
\end{theorem}

As we showed in \cref{lem:equiv}, any $t$-union of the class of equivalence graphs is linearly $\chi$-bounded.
Next, we will show that the complement of any such class is also linearly $\chi$-bounded. This is in contrast with the fact that, in general, the complement of a $\chi$-bounded class is not necessarily $\chi$-bounded (as discussed after \cref{thm:interval-intersection}).
Let $\cX$ be the class of equivalence graphs. Then the complement of the $t$-union of $\cX$ is the $t$-intersection of $\overline{\cX}$, i.e., the $t$-intersection of the class of complete multipartite graphs. 
We start by deriving forbidden induced subgraphs for these classes, which we then use to establish linear $\chi$-boundedness.

\begin{lemma}\label{lem:multipartite forbidden}
    Let $t \in \bN$, and $\cY$ be the $t$-intersection of the class of complete multipartite graphs. Then, $K_{t+1}+ O_1$ and $(2^{t-1}+1)K_2$ are forbidden induced subgraphs for $\cY$.
\end{lemma}
\begin{proof}
    To prove that $H=K_{t+1}+ O_1$ is forbidden for $\cY$, we suppose that $H$ is the intersection of $k$ complete multipartite graphs $G_1, \ldots, G_k$, and show that $k$ should be at least $t+1$. Assume that the vertices of $K_{t+1}$ are labeled $1$ to $t+1$ and the isolated vertex of $H$ is labeled $t+2$.
    Then each $G_i$ must have each of the vertices from $[t+1]$ in a different independent set. So in each $G_i$, $t+2$ can be in an independent set with at most one other vertex from $[t+1]$. As $t+2$ is non-adjacent to all other vertices in $H$, for each vertex $v \in [t+1]$ there must be a $G_i$ where $t+2$ appears in an independent set with $v$. As in each $G_i$, $t+2$ is in an independent set with at most one vertex from $[t+1]$, we conclude that there must be at least $t+1$ graphs in the intersection, i.e., $k\geq t+1$.

    To show that $(2^{t-1}+1)K_2$ is a forbidden induced subgraph for $\cY$, we use the fact that the smallest $t$ such that $\overline{rK_2}$ is a $t$-union of equivalence graphs is $\lceil \log_2(2r)\rceil$ \cite{Alon86}. By taking complements, this means the smallest $t$ such that $rK_2$ is the intersection of $t$ complete multipartite graphs is $\lceil\log_2(2r)\rceil$. The latter implies that $(2^{t-1}+1)K_2$ is a forbidden induced subgraph for $\cY$.
\end{proof}

To prove the desired result, we will use the above \cref{lem:multipartite forbidden} together with the following theorem. 
\begin{theorem}[\cite{Wag80}]\label{thm:mK2-free}
    For any $r \in \bN$, 
    the class of $rK_2$-free graphs is $\chi$-bounded with a $\chi$-binding function $g(x) = x^{2(r-1)}$.
\end{theorem}

\begin{theorem}\label{thm:rK_2+Ks+O1-free}
    Let $s,r \in \bN$, $s \geq 2$, and let $\cY$ be the class of $\{ K_s+O_1, rK_2 \}$-free graphs.
    Then, $\cY$ is linearly $\chi$-bounded with a $\chi$-binding function
    $h(x) = (s-1)^{2(r-1)} \cdot x$.
\end{theorem}
\begin{proof}
    To prove the statement, we will show, by induction on the clique number, that 
    \begin{equation}\label{eq:chi}
        \chi(G) \leq (s-1)^{2(r-1)} \cdot \omega(G)
    \end{equation}
    holds for every $G \in \cY$.
    
    Clearly, if the clique number of $G \in \cY$ is 1, then $\chi(G) = 1 \leq (s-1)^{2(r-1)}$, as required.
    Let $k \geq 2$ and assume that (\ref{eq:chi}) holds for all graphs in $\cY$ with the clique number less than $k$. Let $G=(V,E)$ be an arbitrary graph in $\cY$ with $\omega(G) = k$.
    Let $v$ be a vertex in $G$, and
    let $Q \subseteq V$ be the set of non-neighbors of $v$.
    Note that $G[Q]$ must avoid an induced $K_{s}$ as otherwise such a subgraph together with $v$ would form the forbidden $K_{s}+ O_1$ in $G$. Therefore, the maximum clique size in $G[Q]$ is at most $s-1$ and so, since $G[Q]$ is $rK_2$-free, it follows from \cref{thm:mK2-free} that $\chi(G[Q]) \leq (s-1)^{2(r-1)}$. 
    Furthermore, as $v$ is non-adjacent to $Q$, the graph $G[Q \cup \{v\}]$ can be properly colored with at most $(s-1)^{2(r-1)}$ colors.
    Now, if $N(v) = \emptyset$, then $G = G[Q \cup \{v\}]$ and we have $\chi(G) \leq (s-1)^{2(r-1)} \leq (s-1)^{2(r-1)}\omega(G)$.
    If $N(v) \neq \emptyset$, then $\omega(G[N(v)]) \leq \omega(G)-1$, so
    \[
        \chi(G) \leq \chi(G[Q \cup \{v\}]) + \chi(G[N(v)]) \leq (s-1)^{2(r-1)} + (s-1)^{2(r-1)} \cdot (\omega(G)-1) = (s-1)^{2(r-1)}\omega(G),
    \]
    using the inductive hypothesis.
\end{proof}

As a corollary of \cref{lem:multipartite forbidden} and \cref{thm:rK_2+Ks+O1-free} we finally obtain 
\begin{theorem}\label{thm:multipartite intersection}
    Let $t \in \bN$, and let $\cY$ be the $t$-intersection of the class of complete multipartite graphs. 
    Then $\cY$ is linearly $\chi$-bounded with a $\chi$-binding function
    $h(x) = t^{2^t} \cdot x$.
\end{theorem}

\subsubsection{Subclasses of perfect graphs}\label{sec:chi-perfect}

In this section we reveal special cases when functions of the class of equivalence graphs preserve perfectness. We will use the characterization of perfect graphs in terms of forbidden induced subgraphs. An \emph{odd hole} in a graph is an induced cycle of odd length at least $5$, and an \emph{odd antihole} is an induced complement of a cycle of odd length at least $5$.

\begin{theorem}[Strong Perfect Graph Theorem \cite{CRST06}]\label{thm:strongPerfect}
    A graph is perfect if and only if it has neither odd holes nor odd antiholes.
\end{theorem}

\begin{theorem}\label{lem:2-equivalence}
    Let $f : \zo^2 \rightarrow \zo$ be a Boolean function on two variables and $\cX$ be the class of equivalence graphs. Then $f(\cX)$ is a class of perfect graphs.
\end{theorem}
\begin{proof}
    To prove the statement, we will verify it for each of the 2-variable Boolean functions. There are 16 such functions, which we denote $f_i$ for $0 \leq i \leq 15$ and define below.
    \[
    \begin{array}{c|c||c|c|c|c|c|c|c|c|c|c|c|c|c|c|c|c}
        x & y & f_0 & f_1 & f_2 & f_3 & f_4 & f_5 & f_6 & f_7 & f_8 & f_9 & f_{10} & f_{11} & f_{12} & f_{13} & f_{14} & f_{15} \\ \hline
        0 & 0 & 0   & 0   & 0   & 0   & 0   & 0   & 0   & 0   & 1   & 1   & 1      & 1      & 1      & 1      & 1      & 1      \\
        0 & 1 & 0   & 0   & 0   & 0   & 1   & 1   & 1   & 1   & 0   & 0   & 0      & 0      & 1      & 1      & 1      & 1      \\
        1 & 0 & 0   & 0   & 1   & 1   & 0   & 0   & 1   & 1   & 0   & 0   & 1      & 1      & 0      & 0      & 1      & 1      \\
        1 & 1 & 0   & 1   & 0   & 1   & 0   & 1   & 0   & 1   & 0   & 1   & 0      & 1      & 0      & 1      & 0      & 1      \\
    \end{array}
    \]
    To simplify the analysis, first notice that if a Boolean function $g : \zo^2 \rightarrow \zo$ is the negation of a Boolean function $f : \zo^2 \rightarrow \zo$, i.e., $g(x,y) = \overline{f(x,y)}$ for all $x,y \in \zo$, then for any two graphs $H_1$, $H_2$ on the same vertex set, the graph $g(H_1,H_2)$ is the complement of $f(H_1,H_2)$. Thus, since the complement of a perfect graph is also perfect \cite{Lov72}, we only need to verify the claim for $f_i$, $0 \leq i \leq 7$, as $f_{15-i}$ is the negation of $f_i$ for every $0 \leq i \leq 7$. We proceed by showing that for each $0 \leq i \leq 7$ and for any two equivalence graphs $H_1,H_2$ the graph $G=f_i(H_1,H_2)$ is perfect.
    
    \noindent (1) \textbf{Function $f_0$}. The graph $G$ is empty, which is clearly perfect.

        \medskip
        
    \noindent  (2) \textbf{Function $f_1$}. The graph $G$ is the intersection of $H_1$ and $H_2$, and therefore it is an equivalence graph and thus perfect.

        \medskip
        
   \noindent  (3) \textbf{Functions $f_2$ and $f_4$}. Note that $f_2(x,y) = x \wedge \overline{y}$ and $f_4(x,y) = \overline{x} \wedge y$, and thus each of $f_2(H_1,H_2)$ and $f_4(H_1,H_2)$ is the intersection of an equivalence graph and the complement of an equivalence graph (i.e., a complete multipartite graph).
        From \cref{thm:strongPerfect} and since any induced subgraph of an equivalence graph (respectively, complete multipartite graph) is again an equivalence graph (respectively, complete multipartite graph), in order to verify this case, it is enough to show that no odd hole or odd antihole can be expressed as 
        the intersection of an equivalence graph and a complete multipartite graph.

        Assume that there exists a $k\geq 2$ and two equivalence graphs, $H_1$ and $H_2$, such that $C_{2k+1} = H_1 \wedge \overline{H_2}$. Then the edge set of $C_{2k+1}$ is a subset of the edges of $H_1$ and so $H_1$ is a connected graph. As $H_1$ is an equivalence graph, this means it must be $K_{2k+1}$. Therefore $\overline{H_2}=C_{2k+1}$, which implies that  $\overline{C_{2k+1}}$ is an equivalence graph, which is a contradiction.
    
        Assume now that there exists a $k\geq 2$ and two equivalence graphs, $H_1$ and $H_2$, such that $\overline{C_{2k+1}} = H_1 \wedge \overline{H_2}$. Then $C_{2k+1} = \overline{H_1} \vee H_2$, where $\overline{H_1}$ is a complete multipartite graph. Then the edges of $\overline{H_1}$ must be a subset of the edges of $C_{2k+1}$, and we claim that $\overline{H_1}$ must be the empty graph. Indeed, assume that $\overline{H_1}$ contains an edge between vertices $v$ and $w$ which are adjacent in $C_{2k+1}$. Take a vertex, $z$, that is not adjacent to $v$ or $w$ in $C_{2k+1}$. Then $z$ is not adjacent to $v$ or $w$ in $\overline{H_1}$, which is not possible as $K_2+O_1$ is a forbidden induced subgraph for complete multipartite graphs.
        This proves that $\overline{H_1}$ is the empty graph and implies that $H_2=C_{2k+1}$, which is not possible as $H_2$ is an equivalence graph.

        \medskip

\noindent  (4) \textbf{Functions $f_3$ and $f_5$}. Note that $f_3(x,y)=x$ and $f_5(x,y)=y$. Thus, $f_3(H_1,H_2) = H_1$ and $f_5(H_1,H_2) = H_2$, which, obviously, are perfect graphs.

        \medskip

        \noindent  (5) \textbf{Function $f_6$}. In this case $G = H_1 \xor H_2$. 
        As in case (3), we will show that no odd hole or odd antihole is an XOR of two equivalence graphs.

        Assume that there exists a $k\geq 2$ and two equivalence graphs, $H_1$ and $H_2$, such that $C_{2k+1} = H_1 \xor H_2$. Then each edge of the cycle belongs to exactly one of the two equivalence graphs. Denote by $v_1, v_2, \ldots, v_{2k+1}$ the vertices along the cycle $C_{2k+1}$.
        
        First, we claim that there are two vertex-disjoint edges of the cycle that belong to the same clique in one of the equivalence graphs. Let $s$ be the length of the longest sequence of consecutive vertices of the cycle that belong to the same clique in one of $H_1$ and $H_2$. If $s \geq 4$, then we are done. Otherwise, we should have $s=3$. Indeed, since the cycle has an odd number of edges, there are two consecutive edges that belong to the same equivalence graph, and as they share a vertex, they must be in the same clique.
        Now, without loss of generality, let $v_2,v_3,v_4$ be three consecutive vertices of the cycle that belong to a clique in $H_1$. Then, by assumption, $(v_1,v_2)$ and $(v_4,v_5)$ are adjacent in $H_2$, but not in $H_1$. Furthermore, since $v_2$ and $v_4$ are not adjacent in the  cycle, but are adjacent in $H_1$, they must be adjacent in $H_2$ as well. This implies that $v_1,v_2,v_4,v_5$ form a path in $H_2$, and thus belong to the same clique in $H_2$, which implies the claim. 
 
        Now suppose that $v_{i_1},v_{i_2},v_{i_3},v_{i_4}$ are four distinct vertices that are in the same clique in $H_1$, and $(v_{i_1},v_{i_2})$ and $(v_{i_3},v_{i_4})$ are two edges of the cycle.
        Then at least one of $v_{i_1}$ and $v_{i_2}$, say $v_{i_1}$, is neither adjacent to $v_{i_3}$ nor to $v_{i_4}$ in the cycle. This means that $v_{i_1}$ is adjacent to $v_{i_3}$ and $v_{i_4}$ in both $H_1$ and $H_2$. But then $v_{i_1},v_{i_3},v_{i_4}$ must be in the same clique in $H_2$, and in particular $v_{i_3}$ is adjacent to $v_{i_4}$ in $H_2$, which is not possible as $v_{i_3}$ and $v_{i_4}$ should be adjacent in exactly one of $H_1$ and $H_2$. 

        \medskip
        Assume now that there exists a $k\geq 2$ and two equivalence graphs, $H_1$ and $H_2$, such that $\overline{C_{2k+1}} = H_1 \xor H_2$.
        Then each edge of the $\overline{C_{2k+1}}$ belongs to exactly one of the two equivalence graphs. Denote by $v_1, v_2, \ldots, v_{2k+1}$ the vertices along the cycle $C_{2k+1}$.

        First, since $\overline{C_5} = C_5$, the case of $k=2$ is covered by the previous analysis. Hence, assume $k \geq 3$. Note that neither $H_1$ nor $H_2$ is a complete or empty graph, as otherwise one of them must be $C_{2k+1}$ or $\overline{C_{2k+1}}$, which is not possible due to the perfectness of both graphs. Furthermore, the clique number of at least one of $H_1$ and $H_2$ should be at least 3, as otherwise $H_1$ and $H_2$ would be graphs of maximum degree one, and the total number of edges in $H_1 \xor H_2$ would not exceed $2k$. The latter is not possible as, for $k \geq 3$, the number $\binom{2k+1}{2}-(2k+1)$ of edges in $\overline{C_{2k+1}}$ is more than $2k$.
        Thus, without loss of generality, we can assume that $3 \leq \omega(H_1) \leq 2k $.

        Let $Q$ be a clique of size $\omega(H_1)$ in $H_1$, so $|Q|\geq 3$. We claim that there exist three vertices $v_i, v_j, v_k$ such that $v_i, v_j \in Q$, $v_k \not\in Q$, and no two of these vertices are consecutive vertices on the cycle $C_{2k+1}$.
        Let $s$ be the maximum number of consecutive vertices of $C_{2k+1}$ that belong to $Q$. Without loss of generality, assume that $v_1, v_2, \ldots, v_s$ belong to $Q$. Then $v_{2k+1}$ and $v_{s+1}$, which are not necessarily distinct, are not in $Q$.
        \begin{enumerate}
            \item If $s \geq 5$, then $v_2,v_4,v_{s+1}$ are the sought vertices.

            \item If $s=4$, then $v_1,v_3,v_5$ are the sought vertices.

            \item Let $s=3$. If $v_5 \not\in Q$, then $v_1,v_3,v_5$ are the sought vertices; if $v_5 \in Q$, then $v_2,v_5,v_{2k+1}$ are the sought vertices (note that as $2k+1 \geq 7$ the vertices $v_5,v_{2k+1}$ are not consecutive on the cycle $C_{2k+1}$).

            \item Let $s \leq 2$. Then, since $|Q| \geq 3$, there are two vertices in $Q$ that are not consecutive vertices of the cycle $C_{2k+1}$. Without loss of generality, assume that $v_1$ and $v_t$ are two vertices in $Q$ such that $t \geq 3$ and every vertex $v_{\ell}$, where $2 \leq \ell \leq t-1$, does not belong to $Q$. 
            \begin{enumerate}
                \item If $t \geq 5$, then $v_1,v_t$, and $v_3$ are the sought vertices. 
                
                \item Let $t = 4$. If $v_6 \not\in Q$, then $v_1,v_4,v_6$ are the sought vertices.
                If $v_6 \in Q$, then $v_1,v_6,v_3$ are the sought vertices.

                \item Let $t = 3$. If $v_6 \not\in Q$, then $v_1,v_3,v_6$ are the sought vertices.
                If $v_6 \in Q$, then, due to the assumption that $s \leq 2$, at least one of $v_4$ and $v_5$ is not in $Q$.
                If $v_4 \not\in Q$, then $v_1,v_6,v_4$ are the sought vertices. If $v_5 \not\in Q$, then 
                $v_1,v_3,v_5$ are the sought vertices.
            \end{enumerate}

            This completes the proof of the claim.
        \end{enumerate}
        
        Now, suppose $v_i,v_j,v_k$ are three non-consecutive vertices of the cycle $C_{2k+1}$ such that $v_i,v_j \in Q$ and $v_k \not\in Q$. Since $v_k$ is adjacent to $v_i$ and $v_j$ in $H_1 \xor H_2$ and $v_k$ is not adjacent to $v_i$ or $v_j$ in $H_1$,
        it follows that $v_k$ must be adjacent to $v_i$ and $v_j$ in $H_2$. This implies that $v_i,v_j,v_k$ are in the same clique in $H_2$, and thus $v_i$ and $v_j$ are adjacent in $H_2$. Since $v_i$ and $v_j$ are also adjacent in $H_1$, we arrive at a contradiction, as every edge of $H_1 \xor H_2$ must belong to exactly one of the two graphs.

        \medskip

     \noindent  (6) \textbf{Function $f_7$}. In this case $G = H_1 \vee H_2$.
        As before, to prove that $G$ is perfect, we will show that no odd hole or odd antihole is a union of two equivalence graphs. 

        Assume that there exists a $k\geq 2$ and two equivalence graphs, $H_1$ and $H_2$, on $2k+1$ vertices, such that $C_{2k+1} = H_1 \vee H_2$. As the largest clique in $C_{2k+1}$ has size $2$, every connected component of $H_1$ and $H_2$ has at most $2$ vertices. Therefore, the maximum number of edges that either $H_1$ or $H_2$ could have is $k$, which implies that $H_1 \vee H_2 = C_{2k+1}$ has at most $2k$ edges, a contradiction.

        Assume now that there exists a $k \geq 2$ and two equivalence graphs, $H_1$ and $H_2$, on $2k+1$ vertices such that $\overline{C_{2k+1}} = H_1 \vee H_2$.
        Note that if $k=2$, then $\overline{C_{2k+1}} = \overline{C_5} = C_5$, which as we showed earlier cannot be a union of two equivalence graphs.
        It is known \cite{Alon86} that for any $s \geq 3$, if $\overline{C_s}$ is a union of $t$ equivalence graphs, then $t \geq \log_2 s-1$.
        This implies that $\overline{C_{2k+1}}$ cannot be the union of $2$ equivalence graphs for any $k\geq 4$. 
        It remains to verify the  case of $k=3$, i.e., to show that $\overline{C_{7}}$ is not a union of two equivalence graphs.
        Suppose for a contradiction that there exist two equivalence graphs $H_1$ and $H_2$ such that $\overline{C_7} = H_1 \vee H_2$.
        The largest clique in $\overline{C_7}$ is of size $3$ (as the largest independent set in $C_7$ has size $3$), so each connected component of $H_1$ and $H_2$ has at most $3$ vertices. Since $H_1$ and $H_2$ each have $7$ vertices, this implies that none of them can have more than $6$ edges. Consequently, $H_1 \vee H_2$ has at most $12$ edges. However, $\overline{C_7}$ has $14$ edges, a contradiction.
\end{proof}

\cref{lem:2-equivalence} is best possible in the sense that there exists a 3-variable Boolean function $f$ such that if $\cX$ is the class of equivalence graphs, then $f(\cX)$ is not a subclass of perfect graphs. Indeed, \cref{th:binding lower bound} implies that the 3-XOR of the class of equivalence graphs is not a subclass of perfect graphs as its smallest $\chi$-binding function is at least quadratic.

Determining the smallest $\chi$-binding function of the $k$-union of perfect graphs is a natural question. In \cite[Problem 5.3]{Gya87}, this question was raised for $k=2$. Below, we prove that intersection with the class of equivalence graphs preserves being a class of perfect graphs. We use this result to deduce that the smallest $\chi$-binding function for the 3-union of the class of perfect graphs is $\Omega(x^2)$.

\begin{lemma}\label{lem:f_1}
    Let $\cX$ be the class of equivalence graphs and $\cY$ be a class of perfect graphs. Then $\cX\wedge\cY$ is a class of perfect graphs.
\end{lemma}
\begin{proof}
    Let $F$ be an arbitrary graph in $\cX \wedge \cY$. We need to show that $F$ is a perfect graph. Let $H\in\cX$ and $G\in\cY$ be such that $F = H \wedge G$. Then $F$ is a disjoint union of induced subgraphs of $G$, where each induced subgraph corresponds to a maximal clique of $H$. Since every induced subgraph of a perfect graph is perfect, and a disjoint union of perfect graphs is also a perfect graph, it follows that $F$ is perfect. 
\end{proof}

\begin{lemma}\label{lem:3-union-perfect}
    Let $\cQ$ be the 3-union of the class of perfect graphs.
    Then the smallest $\chi$-binding function for $\cQ$ is $\Omega(x^2)$. 
\end{lemma}
\begin{proof}
    We will show that the 3-XOR of equivalence graphs is contained in $\cQ$, which together with \cref{th:binding lower bound} will imply the result.

    Since $x \xor y = (x \wedge \overline{y}) \vee (\overline{x} \wedge y)$, we have
    \[
        x \xor y \xor z = 
        [ x \wedge (\overline{y \xor z}) ] \vee [ \overline{x} \wedge (y \xor z) ] =
        [ x \wedge (\overline{y \xor z}) ] \vee (\overline{x} \wedge y \wedge \overline{z}) \vee (\overline{x} \wedge \overline{y} \wedge z).
    \]
    Thus, if $H_1, H_2, H_3$ are equivalence graphs, then 
    \[
        H_1 \xor H_2 \xor H_3 = [ H_1 \wedge (\overline{H_2 \xor H_3}) ] \vee [H_2 \wedge (\overline{H_1} \wedge \overline{H_3})] \vee [H_3 \wedge (\overline{H_1} \wedge \overline{H_2})],
    \]
    and therefore by \cref{lem:2-equivalence,lem:f_1}, $H_1 \xor H_2 \xor H_3$ is a union of 3 perfect graphs.
\end{proof}

Since the submission of this paper, it was proved in \cite{BCCDKO26} that the smallest $\chi$-binding function for the $k$-union of comparability graphs is exactly $x^k$ (the authors also note that \cite[Theorem 14]{PT20} established that every $\chi$-binding function for the $2$-union of comparability graphs is at least ${x+1\choose 2}$). Together with \cite[Corollary 5.2]{Gya87} this implies that the smallest $\chi$-binding function for the $k$-union of perfect graphs is exactly $x^k$.

\section{Model-theoretic perspective}
\label{sec:model-theory}

Boolean combinations of graph classes provide a way to produce new graph classes from old. In this section, we take a model-theoretic view of this process, and compare Boolean combinations to transductions, another method for producing new graph classes from old. 

From the model-theoretic point of view, Boolean combinations are a composition of the two operations of expansion and quantifier-free definition. In order to obtain $G$ as a function of graphs from a class $\cX$, we are allowed to choose a graph $H \in \cX$, repeatedly expand it by other graphs from $\cX$ (using a distinct relation symbol for each edge set), and then define the new edge relation using a quantifier-free formula. By contrast, transductions allow only expansions by unary relations, but then allow for the new edge relation to be defined by an arbitrary formula. 

Boolean combinations and first-order transductions are incomparable in their expressive power. For example, first-order transductions can take the square of a graph (connecting two vertices at distance at most 2) by using an existential formula. In particular, the class of all graphs can be first-order transduced from the class of 1-subdivisions of complete graphs. This is not possible to achieve with Boolean combinations as the hereditary closure of the 1-subdivisions of complete graphs has degeneracy 2, and, by \cref{cor:bdd-dgn}, any function of this class is contained in the union $\cY \cup \overline{\cY}$ for some class $\cY$ of bounded degeneracy, which cannot be the case for the class of all graphs.

On the other hand, the class of 2-unions of star forests contains $2$-subdivisions of complete graphs. However, the class of $2$-subdivisions of complete graphs cannot be first-order transduced from the class of star forests as the latter is monadically NIP and the former is not, and first-order transductions preserve the property of being monadically NIP (see e.g.~\cite{braunfeld2025first}).

It seems natural to ask to what extent these operations could be merged. For structures of structurally bounded degree (recall \cref{sec:clique+deg1}), \cite{laskowski2013mutually} (which calls these classes \emph{mutually algebraic}) shows that we may freely overlap these structures on the same vertex set, freely expand by unary relations, and then consider all relations definable by first-order formulas, and the result will still have structurally bounded degree. However, \cite{braunfeld2022worst} shows that this merging cannot be pushed any further: if $\cX$ is a class that is not of structurally bounded degree and $\cY$ is a class that is not definable from unary relations, then there is some way to overlay the structures in $\cX$ and $\cY$ so that the result is not monadically NIP. Furthermore, \cite{braunfeld2022existential} shows that the complexity of classes that are not monadically NIP already appears when considering graphs definable by existential formulas, rather than arbitrary first-order formulas, and that Boolean combinations of $\exists\forall$ suffice to produce all graphs. This places strong restrictions on what sort of formulas we may use in combination with expansions outside of the structurally bounded degree case, although there may be some space to consider, for example, classes existentially defined from overlaps of permutation graphs.

There is also some connection to the notion of set-defined classes from \cite{jiang2020regular}. A class is \emph{set-defined} if it is a subclass of the class of finite induced subgraphs of a graph definable in $(\mathbb{N}, =)$. Since this structure has quantifier elimination, this amounts to saying that there is some $d$ such that every vertex can be associated with an element of $\mathbb{N}^d$, and whether to put an edge between two vertices is determined by some Boolean combination of equalities between the entries of these tuples, where comparisons between entries in different coordinates are allowed. These are precisely the classes of graphs that admit equality-based labeling schemes \cite{HWZ22}. 
It is not hard to show that functions of the class of equivalence graphs are the subset of set-defined classes where comparisons between entries in different coordinates are not allowed. This ensures more structure since, for example, functions of equivalence graphs are $\chi$-bounded while set-defined classes need not be \cite{jiang2020regular}. 

\section{Conclusions}\label{sec:conclusiong}
In this paper we studied Boolean combinations of graphs and how they affect various graph properties.

A natural direction for future research is to study ``Boolean complexity'' of graphs with respect to specific graph classes.
Given a graph $G$ and a class of graphs $\cX$, we say that the \emph{Boolean dimension of $G$ with respect to $\cX$} is the minimum number $k$ such that $G$ is a Boolean combination of $k$ graphs from $\cX$. Using this notion, the fact that a class $\cY$ is a function of a class $\cX$ is equivalent to saying that the elements of $\cY$ have bounded Boolean dimension with respect to $\cX$. Studying which graph classes have bounded Boolean dimension with respect to some well-behaved graph classes may be a helpful tool for deriving useful graph properties about these classes. 

For example, moving from classes of graphs to posets (which can be seen as transitive acyclic directed graphs), one can consider the notion of Boolean dimension for a class of posets with respect to the class of linear orders. This is studied in \cite{nevsetvril1989note,GNT90}, which asks the question of whether the class of planar posets has bounded Boolean dimension with respect to linear orders. Related results and open problems are surveyed in \cite{barrera2020comparing}.

Boolean dimension can further be specified by restricting the type of Boolean functions that are allowed to be used. For example, by restricting ourselves to conjunctions, we obtain the notion of intersection dimension \cite{CozzensR89,KratochvilT94}; and by restricting ourselves to disjunctions, we obtain notions of cover number (see e.g. \cite{GMT24,Mar24}). Restricting to monotone Boolean functions gives a common generalization of the two, which is still not as general as arbitrary Boolean functions.

\bigskip

\textbf{Acknowledgments.}
We are grateful to Nathan Harms for useful discussions on the topic of this paper and for pointing out the results in learning theory showing that stability is preserved under Boolean combinations. We thank a reviewer of this paper for many detailed comments, and most notably for improving the bound in \cref{lem:new-split}.

This work has been supported by the Royal Society (IES\textbackslash R2\textbackslash 242173). This paper is part of a project that has received funding from the 
	European Research Council (ERC) under the European Union's Horizon 2020 
	research and innovation programme (grant agreement No 810115 - Dynasnet). Samuel Braunfeld is further supported by Project 24-12591M of the Czech Science Foundation (GA\v{C}R), and is partly supported by the long-term strategic development financing of the Institute of Computer Science (RVO: 67985807).

\bibliography{references}

\end{document}